\documentclass[11pt,a4paper]{article}
\usepackage{a4}
\usepackage{amsmath}
\usepackage{amsthm}
\usepackage{amssymb}
\usepackage[applemac]{inputenc}
\usepackage{graphicx}
\usepackage{color}
\usepackage{tikz}
\usepackage{enumerate} 
\usepackage{a4wide}
\usepackage{adjustbox}
\usetikzlibrary{matrix}
\def\abs#1{\langle\hspace*{-0.08cm}\cdot
\hspace*{0.02cm} #1\hspace*{0.02cm}\cdot\hspace*{-0.08cm}\rangle}

\newcommand{\sgn}{\operatorname{sgn}}

\newcommand{\cadlag}{{c\`adl\`ag}}

\def\A{{\cal A}}

\def\B{{\cal B}}
\def\C{{\cal C}}

\def\F{{\cal F}}

\def\G{{\cal G}}

\def\H{{\cal H}}
\def\fH{{\mathfrak H}}

\def\J{{\cal J}}

\def\NN{{\mathbb N}}

\def\P{{\cal P}}
\def\Q{{\mathbb Q}}
\def\R{{\mathbb R}}

\def\X{{\cal X}}

\def\Y{{\cal Y}}
\def\Z{{\cal Z}}

\def\PROB{{\mathbb P}}
\def\PP{{\mathbb P}}
\def\EXP{{\mathbb E}}
\def\EE{{\mathbb E}}

\def\IND{{\mathbb I}}

\def\sw{{\bar{\omega}}}

\def\omg{{\omega}}

\newtheorem{theorem}{Theorem}[section]

\newtheorem{corollary}[theorem]{Corollary}

\newtheorem{lemma}[theorem]{Lemma}

\newtheorem{proposition}[theorem]{Proposition}

\newtheorem{definition}[theorem]{Definition}

\newtheorem{assumption}[theorem]{Assumption}
\newtheorem{remark}[theorem]{Remark}

\numberwithin{equation}{section}
\title{Causal optimal transport and its links to enlargement of filtrations and continuous-time stochastic optimization}

\author{B. Acciaio\thanks{The London School of Economics and Political Science, Dept. Statistics, 10 Houghton St, WC2A 2AE London, UK. Email: b.acciaio@lse.ac.uk} \and J. Backhoff Veraguas\thanks{Vienna University of Technology, Institute of Statistics and Mathematical Methods in Economics (E105-7), Wiedner Hauptstra\ss e 8-10, A-1040 Vienna, Austria. Email: julio.backhoff@tuwien.ac.at
} \and A. Zalashko\thanks{University of Vienna, Faculty of Mathematics, Oskar-Morgenstern-Platz 1, A-1090 Vienna, Austria. Email: anastasiia.zalashko@univie.ac.at. The financial support by the Doktoratskolleg W1245 of the Austrian Science Fund (FWF) and by the Department of Statistics, LSE is gratefully acknowledged.}}

\begin{document}
\maketitle

\begin{abstract}
The martingale part in the semimartingale decomposition of a Brownian motion with respect to an enlargement of its filtration, is an anticipative mapping of the given
Brownian motion. In analogy to optimal transport theory, we define causal transport plans in the context of enlargement of filtrations, as the Kantorovich counterparts of the aforementioned non-adapted mappings. We provide a necessary and sufficient condition for a Brownian motion to remain a semimartingale in an enlarged filtration, in terms of certain minimization problems over sets of causal transport plans. The latter are also used in order to give robust transport-based estimates for the value of having additional information, as well as model sensitivity with respect to the reference measure, for the classical stochastic optimization problems of utility maximization and optimal stopping.  Our results have natural extensions to the case of general multidimensional continuous semimartingales.
\\[1ex]

\\[2ex]
\noindent{\textbf 
{MSC2010 subject classifications:} 91G80, 60G44, 90C08.
}
\end{abstract}

\section{Introduction}\label{sec Introduction}
From the seminal works of Monge~\cite{Mo81} and Kantorovich~\cite{Ka42}, the theory of optimal transport has widely developed and established itself as a fervent research area, with growing applications in the most various areas of sciences and engineering. 
Powerful connections have also been established between the theory of optimal transport and stochastic analysis, including among many others, the work by Feyel and {\"U}st{\"u}nel~\cite{FeyelUstunel} extending Brenier's result to Wiener spaces, \cite{BHP13,GHT14} on model-independent finance, and \cite{BCH} on Skorokhod Embedding.
In the recent article by Lassalle~\cite{Lassalle2} the author creates another bridge between optimal transport and stochastic analysis, considering the transport problem under the so called \emph{causality} constraint. The origins of this concept can be found in the work of Yamada and Watanabe \cite{YW}; see also \cite{Jac81,Kurtz2014} for a generalization of the latter. For a discrete-time analogue of transport under causality, see \cite{BBLZ}.

The aim of the present article is to exploit ideas and techniques from optimal transport under causality, in order to revisit the classical stochastic analysis problem of \emph{enlargement of filtrations}. We recall that the central question of enlargements of filtrations is whether the semimartingale property is preserved when passing from a given filtration to a larger one; see \cite{BY78,JY78,Jeu80,Jac85} for some of the earliest works on the subject. 
{ We also stress that, from the point of view of financial applications, considering different filtrations means accommodating agents having access to different sets of information. {This clearly triggers the 
question, of how much having different information matters when facing a particular optimization/decision problem.} By means of causal transport, we will address both the issue of semimartingale preservation, as well as { that of estimating} the value of different (usually additional) information.}
To describe causality, one is first given two Polish filtered probability spaces $(\X,\{\F^\X_t\}_{t=0}^T,\mu)$ and $ (\Y,\{\F^\Y_t\}_{t=0}^T,\nu)$. A transport plan $\pi$ is a probability measure on $\X\times\Y$ having the prescribed marginals $\mu,\nu$; this is denoted by $\pi\in\Pi(\mu,\nu)$.
It is further called \emph{causal} if a certain measurability condition holds, roughly: 
the amount of ``mass'' transported by $\pi$ to a subset of the target space $\Y$ belonging to $\F^\Y_t$, is solely determined by the information contained in $\F_t^\X$. Thus a causal plan transports $\mu$ into $\nu$ in an adapted way. 
Although Lassalle analyzes this constrained transport problem in a general set-up,  his most noteworthy results (e.g.\ connection between relative entropy, weak solutions of stochastic differential equations, and causal transports) are obtained in the setting of $\X=\Y$ being the space of continuous functions and, importantly, both filtrations $\F_t^\X = \F_t^\Y$ being the canonical one; see \cite[Sect. 6]{Lassalle2}. This framework does not allow for anticipation of information, so it is not suitable for the study of enlargement of  filtrations.

Given a cost function $c$ on $\X\times\Y$, the general \emph{causal transport problem} is defined as 
\begin{equation}\label{eq:ctp_intro}
\inf\{\EE^\pi[\, c \, ] : \pi\in\Pi(\mu,\nu),\, \text{$\pi$ causal}\}.
\end{equation} 
The situation of interest for our purposes is when both $\X $ and $ \Y$ are the space of continuous functions, possibly endowed with different filtrations.
Concretely, let $B$ be a Brownian motion on some probability space $(\Omega,\F^B,\PROB)$, where $\F^B$ is the filtration generated by $B$, and let $\H\supseteq\F^B$ be a finer filtration (i.e., $\H$ is an \emph{enlargement} of $\F^B$). If $B$ is still a semimartingale with respect to the larger filtration $\H$, then its unique continuous semimartingale decomposition takes the form
\begin{equation}\label{eq:dec_intro}
dB_t=d \tilde B_t + dA_t,
\end{equation}
where $\tilde B$ is an $\H$-Brownian motion and $A$ is a continuous $\H$-adapted finite variation process. Then the joint law of $(\tilde B, B)$ turns out to be a causal transport plan on path space, when considering the canonical and an appropriate enlarged filtration (see Section~\ref{sect: CFG} for the precise framework). Since $\tilde B$ is an anticipative but deterministic mapping of $B$, much as a Monge map in classical transport (but mapping a target measure to the source one), one can say that causal transport plans { correspond to a} Kantorovich generalization of such anticipative mappings.

The main theoretical result of this article is a characterization of the preservation of the semimartingale property in an enlarged filtration, for a process which is a Brownian motion in the original filtration. A necessary and sufficient condition for this preservation property is given in terms of the causal transport problem \eqref{eq:ctp_intro} on continuous path space, for specific cost functions depending on the difference of the coordinate processes on the product space; see Theorem~\ref{thm: DM_new}.
In addition, when considering transport plans under which this difference is absolutely continuous with respect to Lebesgue measure, we can give necessary and sufficient conditions not only for the semimartingale preservation property to hold, but also to ensure that the finite variation process in \eqref{eq:dec_intro} is absolutely continuous (which yields the so-called information drift); see Theorem~\ref{thm: DM_new_2}. 
When the cost function is of Cameron-Martin type, and the filtration enlargement is done entirely at time zero, the causal transport problem can be interpreted in terms of entropy and mutual information. Thus we are inclined to say that, irrespective of the cost function and the kind of enlargement, the value of our causal transport problems can be seen as a mutual information in a wider sense.
A generalization of the definition of causality allows us to determine necessary and sufficient conditions for a general continuous semimartingale to remain a semimartingale with respect to an enlarged filtration.

Another contribution of the article consists in the analysis of duality for the \emph{primal} problem \eqref{eq:ctp_intro}. Notoriously, duality plays a central role in classical optimal transport. We formulate the causality property via infinitely many linear constraints, which naturally leads to the formulation of a dual problem for \eqref{eq:ctp_intro}.
In order to prove that the values of the primal and dual coincide, we cannot invoke existing results, as these would require imposing restrictive conditions on the problem, and we shall rather take advantage of the specific setting we work in. In the absolutely continuous case described above, we further identify a non-linear dual problem which we can fully solve and relate to the original  problem \eqref{eq:ctp_intro}. This is novel even in the absence of anticipation/enlargements. Interestingly, this gives a different proof of the semimartingale preservation property, and is achieved through optimal transport and convex analysis techniques, without resorting to stochastic analysis arguments; see Theorem~\ref{thm orlicz}.

We finally describe the main application of the present work. The connection between stochastic analysis and causal transport, as developed in this article, allows us to give a novel application in the framework of continuous-time stochastic optimization. For such problems, we derive what we call \emph{robust transport bounds}. Concretely, we show how causal transport provides robust estimates, for a class of stochastic optimization problems, regarding both
\begin{enumerate}
\item[(i)] the value of additional information, and 
\item[(ii)] model sensitivity. 
\end{enumerate} 
Point (i) refers to the difference between the values of a stochastic optimization problem when the optimization is run over a smaller ``original'' filtration, and when it is done over a finer ``enlarged'' one. We establish that for utility maximization (more generally, stochastic control of linear systems) and optimal stopping, this difference is bounded in a robust way by the value of a causal transport problem{; see Proposition~\ref{prop: ut} and Proposition~\ref{prop:OS}(i)}. On the other hand, Point (ii) refers to the difference between the values of a stochastic optimization problem when the optimization is run under two different probabilistic models (i.e.\ reference probability measures). As in the previous case, we establish that such difference is dominated by a causal transport problem in a rather robust fashion{; see Proposition~\ref{prop:OS}(ii)}. 
{ We refer to \cite{PikovKaratz} for original motivation regarding the value of information, and to  \cite{Pflug,PflugPichler} for a discrete-time approach related to ours.}

The article is organized as follows. In Section~\ref{sec Setting} we introduce the main concepts and present some preliminary results. In Section~\ref{sec semimart Causal} we state and prove our main results on the semimartingale preservation property, we state a  non-linear duality result, and provide some links to the literature. Section~\ref{sect:app} contains applications of causal transport to continuous-time stochastic optimization problems. Then Section~\ref{sect:dual} is devoted to an  in-depth study of duality when viewing \eqref{eq:ctp_intro} as a linear program. Finally, in the appendix we collect some technical results, recall some needed tools, and provide some pending proofs.\\

\noindent{\bf Notation.}\ For a Polish space $\Z$, we use $\P(\Z)$ to denote the Borel probability measures on $\Z$, and endow it with the weak convergence of measures.
Given a probability space $(\Omega, \A, \PP)$ and a measurable map $f:\Omega\to\Z$, $f_\#\PP\in\P(\Z)$ denotes the push forward of $\PP$ by $f$. The symbol $\EE^{\PP}$ denotes integration w.r.t.\ $\PP$. By $\B(\Z)$ we denote the Borel $\sigma$-field on $\Z$, and for any $\sigma$-field $\J\subseteq\B(\Z)$ we write $B(\Z,\J)$ (resp.\ $B_b(\Z,\J)$) for the set of all real-valued functions on $\Z$ that are measurable (resp.\ bounded measurable) w.r.t.\ $\J$. Furthermore, given a measure $\eta\in\P(\Z)$, we denote by $\J^\eta$ the completion of $\J$ w.r.t.\ $\eta$; the unique extension of $\eta$ to $\B(\Z)^\eta$ is still denoted by $\eta$. With $\C(\Z)$ (resp.\ $\C_b(\Z)$) we mean the set of all continuous (resp.\ bounded continuous) real-valued functions defined on $\Z$.

\section{Setting and preliminary results}\label{sec Setting}

\subsection{Classical and causal transport.}\label{sec sub abstract}
Let $(\X,\B(\X),\mu)$ and $(\Y,\B(\Y),\nu)$ be two Polish probability spaces. We denote by $\Pi(\mu,\nu)$ the subset of elements in $\P(\X \times \Y)$ having marginals $\mu$ and $\nu$. The classical optimal transport problem consists in minimizing the cost of transporting the (source) measure $\mu$ to the (target) measure $\nu$, with respect to a given cost function $c:\X\times\Y\to\R\cup\{+\infty\}$. The transportation is represented mathematically by a measure $\pi\in\Pi(\mu,\nu)$, referred to as ``transport plan between $\mu$ and $\nu$,'' so the minimization problem is formulated as
\begin{equation}\label{eq classical transport}
\inf\{\EXP^{\pi} [\, c \, ]\ :\ \pi\in\Pi(\mu,\nu)\}.
\end{equation}
This kind of problems have a rich theory, particularly concerning optimality conditions and duality. The latter means the equivalence between \eqref{eq classical transport} and the following maximization problem
\[\textstyle
\sup \left\{\int \phi d\mu+\int \psi d\nu : \phi\in\C_b(\X),\,\psi\in\C_b(\Y),  \,\phi \oplus\psi\leq c\right\},
\]
where here and throughout the article we let $\phi \oplus\psi\leq c$ stand for ``$\phi(x) +\psi(y)\leq c(x,y)\ \forall x,y$''. We do not give an exhaustive list of references on the matter, but rather recommend \cite{Villani,KellererDuality} for a sample of results in this direction, going from the cost $c$ being lower semicontinuous, to being finite and Borel measurable (and beyond). 
For our purposes, we will first need to slightly extend some of the well-known results in classical transport to our particular setting; see Section~\ref{sect:dualcl}.

We now proceed to introduce the specific class of transport plans we shall consider in this work. To this end, we fix a finite time horizon $[0,T]$, and endow the Polish spaces $\X$ and $\Y$ with right-continuous filtrations $\F^\X=(\F^\X_t)_{ t \in [0,T]}$ and $\F^\Y=(\F^\Y_t)_{ t \in [0,T]}$, with $\F^\X_T=\mathcal{B}(\X)$ and $\F^\Y_T=\mathcal{B}(\Y)$. 
As a rule, $x,y$ will denote generic elements of $\X,\Y$ respectively. 
In the following, we denote by $\F^{\X,\mu}$ the usual $\mu$-completed filtration containing $\F^\X$, with analogous notation throughout the article.

\begin{definition}[Causal transport plan]\label{def:causality}
A transport plan $\pi \in \Pi (\mu, \nu)$ is called \emph{causal} between $(\X,\F^\X,\mu)$ and $(\Y,\F^\Y,\nu)$ if, for any $t\in [0,T]$ and any set $A \in \F^\Y_t$, the map 
\[ \X\ni x \mapsto  \pi^x(A)=\EXP^{\pi}[\IND_A|\F_T^\X](x)\]
is measurable with respect to $\F^{\X,\mu}_t$, where
$\pi^x(dy):=\pi(\{x\}\times dy)$ is { a regular conditional kernel} of $\pi$ w.r.t.\ the first coordinate. Denote the set of such plans by $\Pi^{\F^\X, \F^\Y} (\mu, \nu)$.
\end{definition}
This concept goes back to the so-called Yamada-Watanabe criterion (see \cite{YW}) and has been recently popularized by \cite{Lassalle2}. We stress that the last author actually uses a weaker definition but gives sufficient conditions for its equivalence with the one we give. These are however too restrictive for the purpose of the present work.
In the following remark we collect some useful equivalent characterizations of causality.
We recall that, given two filtrations $\A^1\subseteq\A^2$, the so called $H$-hypothesis holds if every square integrable $\A^1$-martingale is a square integrable $\A^2$-martingale.

\begin{remark}
The set $\Pi^{\F^\X, \F^\Y} (\mu, \nu)$ is never empty, since it contains the product measure.
\end{remark}

\begin{remark}\label{rem: causal}
For a probability measure $\pi \in  \Pi (\mu, \nu)$, the following are equivalent:\\[-0.6cm]
\begin{enumerate}
\item $\pi$ is a causal transport plan w.r.t.\ $\F^\X$ and $\F^\Y$;
\item $\pi\left(\X\times D_t|\F^\X_t\otimes\{\emptyset,\Y\}\right)=\pi\left(\X\times D_t|\F^\X_T\otimes\{\emptyset,\Y\}\right)$\ $\pi$-a.s., for all $t\in[0,T]$, $D_t\in\F_t^\Y$;
\item the $\sigma$-fields $\{\emptyset,\X\}\otimes\F^\Y_t$ and $\F^\X_T\otimes\{\emptyset,\Y\}$ are conditionally independent with respect to $\pi$ given $\F^\X_t\otimes\{\emptyset,\Y\}$, for all $t\in[0,T]$;
\item the $H$-hypothesis holds between {$\F^{\X,\mu}\otimes\{\emptyset,\Y\}$ and $\F^{\X,\mu}\otimes\F^\Y$ }with respect to $\pi$.
\end{enumerate}
The equivalences above can be shown as in \cite[Theorem~3]{BY78}. For convenience of the reader we just stress the reason why causality/conditional independence implies the $H$-hypothesis: if $M$ is a $(\pi,\F^{\X,\mu}\otimes\{\emptyset,\Y\})$-martingale, then $\EE^{\pi}[M_{t+s}|\F^{\X,\mu}_t\otimes\F^\Y_t]=\EE^{\pi}[M_{t+s}|\F^{\X,\mu}_t\otimes\{\emptyset,\Y\}]=M_t$, hence $M$ is a $(\pi,\F^{\X,\mu}\otimes\F^\Y)$-martingale.
\end{remark}

In analogy with \eqref{eq classical transport}, and as in \cite{Lassalle2}, we define the \emph{causal transport problem}:
\begin{equation}\label{eq causal transport}\textstyle
\inf\{\EXP^{\pi} [\, c \, ]\ :\ \pi\in\Pi^{\F^\X,\F^\Y}(\mu,\nu)\},
\end{equation}
which constitutes the core of our work. 
In Section~\ref{sect:dual} we will prove a duality result for this problem, and in Sections~\ref{sec semimart Causal} and \ref{sect:app} we will show how the causal transport problem for specific cost functions allows to characterize semimartingale preservation under filtration enlargement, and to estimate the value of additional information for some stochastic optimization problems.

\subsection{Path space and filtration enlargement.}\label{sect: CFG}
We will consider causal plans that transport measures defined on spaces of continuous functions. 
For $t\in(0,T]$, we denote by $\C [0,t]$ the set of continuous functions $f:[0,t]\to\R$ such that $f(0)=0$, and we let $\C =\C [0,T]$. Let $W=(W_t)_{t\in[0,T]}$ be the coordinate process on $\C$, $W_t(\omega)=\omega_t$ for $\omega\in\C$, and let ${ \F=}(\F_t)_{t\in[0,T]}$ the right-continuous version of the filtration generated by $W$:
\[
\F_t:=\bigcap\limits_{u>t}\sigma(W_s, s\leq u).
\]
In order to consider {all possible kinds of \emph{anticipation of information} regarding the evolution of the coordinate process, we study right-continuous filtrations $\G=(\G_t)_{t\in[0,T]}$ such that
\begin{equation}\label{eq: f_enl}
\G_t\supseteq\F_t\; \textrm{for all $t\in[0,T)$},\;\;\, \textrm{and\;\, $\G_T=\F_T$}. 
\end{equation}}
It is worth mentioning the two most studied kinds of filtration enlargements, which are particular cases of \eqref{eq: f_enl}:\\[-0.6cm]
\begin{itemize}
\item \emph{initial enlargement} with a ($\F_T$-measurable) random variable, say $L$ (so { $\G_t=\F_t\vee\sigma(L)$}, $t\in[0,T]$);
\item \emph{progressive enlargement with a random time} (non-negative $\F_T$-measurable random variable), say $\tau$ (so {$(\G_t)$ is the right-continuous version of $(\G^0_t)$, where $\G^0_t:=\F_t\vee\sigma(\tau\wedge t)$}, $t\in[0,T]$), in which case $\G$ turns $\tau$ into a stopping time.
\end{itemize}
In fact, not many works are devoted to the study of general enlargements of filtration beyond these two cases, as considered in the present article, see e.g. \cite{ADIGirsanov}, { \cite{Jeu80}} and \cite{KchiaProtter}. 
We refer the reader to the monographs \cite{Jeu80}, \cite{MY06}, \cite[Sec.~5.9]{JYC09}, \cite[Ch.~VI]{Pro04} for an account of the main results and the literature on filtration enlargements.

In what follows we will consider $\X=\Y=\C$ and, given two measures $\mu,\nu$ on $\C$, we will study causal transport plans between $(\C,\F,\mu)$ and $(\C,\G,\nu)$. We shall commonly denote by $(\omega,\sw)$ generic elements in $\C\times\C$. 
Often, as source measure $\mu$, we will take the Wiener measure on $\C$, which we denote by $\gamma$.
For a continuous process { $Z=(Z_t)_{t\in[0,T]}$ defined on a given space $\Omega$}, we denote by $\F^Z:=Z^{-1}(\F)$ the right-continuous version of the filtration generated by $Z$ on $\Omega$, and by { $\G^{Z}:=Z^{-1}(\G)$ the filtration containing anticipation of information regarding the evolution of $Z$}.
Working with the path space $\C$ eases the exposition of our analysis. We point out, however, that our results have a natural extension in the multidimensional setting $\C^d$, i.e.\ for multidimensional continuous processes, see Remarks~\ref{rem multi}, \ref{rem multi abs}.

{ For the rest of this section, we work on a fixed probability space $(\Omega,\J,\PP)$. In particular, for any process $Z=(Z_t)_{t\in[0,T]}$ defined on it, it is implicitly understood that $\F^Z_T\subseteq\J$.}

\begin{definition}[Causal coupling]\label{def: CC}
A pair $(X,Y)$ of continuous processes on {  $(\Omega,\J,\PP)$}, is called a \emph{causal coupling} w.r.t.\ the filtrations $\F^X$ and {$\G^{Y}$}\ if $(X,Y)_\#\PROB$ is a causal transport plan between $(\C,\F,X_\#\PROB)$ and $(\C,\G,Y_\#\PROB)$.
\end{definition}

\begin{remark}\label{rem: CC}
Definition~\ref{def:causality} and Remark~\ref{rem: causal} imply that $(X,Y)$ is a causal coupling if and only if for all $t$, {$\G^{Y}_t$} and $\F^X_T$ are conditionally independent given $\F^X_t$, that is, the $H$-hypothesis holds between $\F^X$ and $\F^{X}\vee{ \G^{Y}}$, with respect to $\PROB$.
\end{remark}
The following result shows that there is an easier way to check causality in a Brownian setting; it extends { the result} \cite[Proposition 4]{Lassalle2} to our setting with enlargements.
\begin{lemma}\label{lem: causal wiener equivalence}
Let $X$ be a Brownian motion in its right-continuous natural filtration, and $Y$ be a continuous process. Then the following are equivalent:
\begin{enumerate}
\item $(X,Y)$ is a causal coupling w.r.t.\ $\F^X$ and ${ \G^{Y}}$;
\item $X$ is a Brownian motion in $\F^{X}\vee{ \G^{Y}}$;
\item there is a filtration $\A$ on $\Omega$ s.t.\ $X$ is a $\A$-Brownian motion and {${ \G^{Y}}\subseteq\A$}.
\end{enumerate}
\end{lemma}
{ From the proof of the lemma it will become clear that if $X$ is any process with increments independent  with respect to ${ \G^{Y}}\vee\F^X$, then $(X,Y)$ is causal.}
\begin{proof}
1$\Rightarrow$2:\ { By L\'evy theorem, it is enough} to show that $X$ is a $\F^{X}\vee{ \G^{Y}}$-martingale. For $0\leq s<t\leq T$ and $f_s\in L^\infty(\F^{X}_s\vee{ \G^{Y}_s})$, we have
\begin{eqnarray*}
\EE\left[X_tf_s\right]&=&\EE\left[\EE\left[X_tf_s \mid(X_u)_{u\leq t}\right]\right]=\EE\left[X_t\EE\left[f_s\mid(X_u)_{u\leq t}\right]\right]=\EE\left[X_t\EE\left[f_s\mid(X_u)_{u\leq s}\right]\right]\\
&=&\EE\left[\EE\left[X_t\mid(X_u)_{u\leq s}\right]\EE\left[f_s\mid(X_u)_{u\leq s}\right]\right]
=\EE\left[X_s\EE\left[f_s\mid(X_u)_{u\leq s}\right]\right]=\EE\left[X_sf_s\right],
\end{eqnarray*}
where causality is used in the third equality.\\ 
2$\Rightarrow$3\ follows by taking $\A:=\F^{X}\vee{\G^{Y}}$.\\ 
3$\Rightarrow$1:\ For $t\in[0,T]$ and $D_t\in { \G^{Y}_t}$,
\[\PROB(D_t\ |\ (X_s)_{s\leq T})=\PROB(D_t\ |\ (X_s)_{s\leq t},\, \{X_{t+u}-X_t, u\in[0,T-t]\})=\PROB(D_t\ |(X_s)_{s\leq t}),
\]
since $X$ is a $\F^{X}\vee{ \G^{Y}}$-Brownian motion, hence its increments $X_{t+s}-X_{t}$ are jointly independent of $(X_s)_{s\leq t}$ and of the event $D_t\in { \G^{Y}_t}$. This shows that $(X,Y)$ is causal, by Definition~\ref{def: CC} and Remark~\ref{rem: causal}-2.
\end{proof}

Throughout the article we talk of \emph{(continuous) semimartingale decomposition}, referring to the unique decomposition of a continuous semimartingale into a continuous local martingale and a continuous finite variation process.
The notion of causality can be used to study  semimartingale decompositions in the setting of enlargement of filtrations, and this is the object of study of Section~\ref{sec semimart Causal}. We start by illustrating in Section~\ref{ex:BB} the lemma above, and show a first connection between  decomposition of semimartingales and causality. A necessary and sufficient condition for a Brownian motion to remain a semimartingale in the enlarged filtration is given in Theorem~\ref{thm: DM_new}.

\subsubsection{Some examples of enlargement of Brownian filtration}\label{ex:BB}
In this section we collect some well-known examples of filtration enlargements in a Brownian setting, which will be useful for future reference (see e.g. \cite{MY06} for these and many other examples).
{Let $B$ be a Brownian motion in its right-continuous natural filtration.} If $B$ remains a semimartingale with respect to the enlarged filtration ${ \G^{B}}$, then its unique continuous semimartingale decomposition takes the form
\[
dB_t=d \tilde B_t + dA_t,
\]
where $\tilde B$ is a ${ \G^{B}}$-Brownian motion and $A$ is a continuous ${ \G^{B}}$-adapted finite variation process. In particular, for every finite horizon $T>0$, by Lemma~\ref{lem: causal wiener equivalence}, we have that $(\tilde B,B)$ is a causal transport plan between $\F^{\tilde B}$ and ${ \G^{B}}$, that is, $(\tilde B,B)_\#\PROB\in\Pi^{\F,\G}(\gamma,\gamma)$.

In what follows we recall specific enlargements of the filtration $\F^B$, which will be referred to later on in the article.

$(1)$\ \emph{Initial enlargement with countably many atoms (see \cite{Yor85})}:\ initial enlargement with a discrete { $\F^B_T$-measurable} random variable, say $L$ (namely {$\G^B_t=\F^B_t\vee\sigma(L)$} for all $t\in[0,T]$), that takes values { $l_n$, $n\in\NN$. This corresponds to enlarging the filtration at time zero with the sets $C_n=\{L=l_n\}, n\in\NN$. In this case the decomposition of $B$ in the enlarged filtration takes the form
\begin{equation}\label{eq:atoms}\textstyle
dB_t = d\tilde{B}_t + \sum\limits_n\IND_{C_n}\frac{\eta^n_t}{M^n_t}dt,
\end{equation}
where $M^n_t=\PP(C_n|\F^B_t)$, which by martingale representation can be written as $M^n_t=M^n_0+\int_0^t\eta^n_sdB_s$, with $M^n_0=\PP(C_n)$ and for some predictable process $\eta^n$.}

$(2)$\ \emph{Brownian bridge}:\ initial enlargement by the value of the Brownian motion at the terminal time $T$, namely {$\G^B_t=\F^B_t\vee\sigma(B_T)$} for all $t\in[0,T]$. In this case it is well-known that the decomposition of $B$ in the enlarged filtration ${\G^B} = \F^B \vee \sigma(B_T)$ is 
\begin{equation}\label{eq:Brownian bridge}\textstyle
dB_t = d\tilde{B}_t + \frac{B_T - B_t}{ T-t}dt.
\end{equation}

$(3)$\ \emph{Progressive enlargement with last hitting time} (see \cite[Section 12.2.4]{Yor97}):\
progressive enlargement with the random time $\tau = \sup\left\lbrace  u \leq T, B_u = 0\right\rbrace $. In this case, using the notation $\textstyle\Phi(x)=\sqrt{{2}/{\pi}}\int_x^\infty e^{-u^2/2}du$, the decomposition of $B$ in the enlarged filtration is the following:
\begin{equation}\label{eq:progressive}\textstyle
dB_t = d\tilde{B}_t - \sqrt{\frac{2}{\pi}}\frac{\exp\left({-B_t^2}/{2(T-t)}\right) }{\sqrt{T-t}}\left(\frac{\sgn(B_t)}{\Phi(|B_t|/\sqrt{T-t})}\IND_{\left\lbrace t \leq \tau \right\rbrace } - \frac{\sgn(B_T)}{1-\Phi(|B_t|/\sqrt{T-t})}\IND_{\left\lbrace  \tau \leq t \leq T \right\rbrace }\right)dt.
\end{equation}

$(4)$\ \emph{Bessel process} {(see \cite[Sect. 6.3]{Jeu80})}:\ define a $3$-dimensional Bessel process by $dR_t=\frac{1}{R_t}dt+dB_t$, and denote $J_t:=\inf_{s\geq t}R_s$. Then the process $\tilde B_t:=R_t-2J_t$ is a Brownian motion in the filtration obtained by enlarging $\F^B$ with the process $J$, and
\[\textstyle
dB_t=d\tilde B_t+(2dJ_t-\frac{1}{R_t}dt).
\]
Note that, contrary to the previous cases, here the finite variation process in the semimartingale decomposition of $B$ in the enlarged filtration is not absolutely continuous with respect to Lebesgue measure.

\section{Causal optimal transport and semimartingale decomposition}\label{sec semimart Causal}

Throughout this whole section we consider the continuous path space framework of Section~\ref{sect: CFG}. In particular, we consider the filtrations $\F$ and $\G$ defined there. The following notation will be frequently used: for a process/path $Z$ we denote by $$\textstyle V_t(Z)\,\, = \sup_{0\leq t_1\leq\dots\leq t_k\leq t}\,\,\sum_{i<k} |Z_{t_i}-Z_{t_{i+1}}|$$  the variation of $Z$ up to time $t$. {In Remark~\ref{rem multi} we explain how the next results and arguments are trivially extended to the case of multidimensional continuous paths, see also Remark~\ref{rem multi abs}.}

We show the connection between the semimartingale decompositions arising in enlargement of filtrations, and certain causal optimal transport problems. We consider both the case when the finite variation part in the semimartingale decomposition is absolutely continuous and when it is not; we do it separately for the sake of applications and connection with the literature. All results are shown in a one-dimensional Brownian setting, for simplicity of exposition. {We note, however, that a weaker definition of causality leads to analogous results for general continuous semimartingales, see Remark~\ref{rem semi}.}
Recall that $\gamma$ is the Wiener measure on the path space $\C$, and that $(\omega,\sw)$ denotes a generic element in $\C\times\C$. By $id$ we mean the identity mapping on $\C$.

\subsection{The general case}
We first need to obtain the following result, reminiscent of \cite[Proposition 6]{Lassalle2}:

\begin{theorem}\label{thm: gWnu_gen}
Let $\nu$ be a measure on $\C$ such that $\nu\ll\gamma$. 
Then the following are equivalent:
\begin{itemize}
\item[(i)] for some continuous, $\G^{\nu}$-adapted, integrable variation process $A=A(\sw)$, i.e.\ such that 
\begin{eqnarray}\label{eq: FV}
\EXP^\nu[ V_T(A)]< \infty,
\end{eqnarray}
the process $\xi_t(\sw) :=\sw_t-A_t(\sw)$ is a $(\nu,\G^{\nu})$-Brownian motion;
\item[(ii)] the following causal optimal transport problem is finite: 
\begin{equation}\label{eq: finiteness_gen}
\inf\limits_{\pi\in\Pi^{\F,\G }(\gamma, \nu)}\EXP^\pi[
V_T(\sw-\omega)].
\end{equation}
\end{itemize}
Moreover, whenever (i)-(ii) hold, we have that:\\
1.\ the transport plan $\hat{\pi}:= (\xi,id)_{\#}\nu $ belongs to $\Pi^{\F ,\G }(\gamma, \nu)$, and is optimal for \eqref{eq: finiteness_gen};\\
2.\ for every transport plan $\pi\in\Pi^{\F ,\G }(\gamma, \nu)$ with finite cost in \eqref{eq: finiteness_gen}, the process $\tilde A(\omega,\sw):=A(\sw)$ in (i) is the $(\pi,\{\{\emptyset,\C\}\otimes \G\}^{\pi})$-dual predictable projection of the process $(\omega,\sw)\mapsto \Lambda(\omega,\sw):=\sw-\omega$.
\end{theorem}

Note that Lemma~\ref{lem:cpt} and Theorem~\ref{prop duality} apply, since $\gamma$ satisfies the weak continuity property~\eqref{eq:weak_cont}, and the total variation of the difference of the coordinate processes is bounded from below and lower semicontinuous w.r.t.\ supremum norm. Therefore, $\Pi^{\F,\G }(\gamma, \nu)$ is weakly compact, the problem in \eqref{eq: finiteness_gen} is attained, and duality holds. In the ensuing proof, attainability is also established but via stochastic analysis arguments.

\begin{proof}\leavevmode
$(i)\ \Rightarrow (ii)$:\ From the process $\xi$ in $(i)$, define the coupling $\hat{\pi}:= (\xi,id)_{\#}\nu $. The fact that $\hat{\pi}\in\Pi^{\F ,\G }(\gamma, \nu)$ follows as in the proof of Lemma~\ref{lem: causal wiener equivalence}, and $\EE^{\hat \pi}[
V_T(\sw-\omega)]=\EE^\nu[V_T(A)(\sw)]< \infty$.

$(ii)\ \Rightarrow (i)$:\ 
Fix $\pi\in\Pi^{\F ,\G }(\gamma, \nu)$ with $\EE^\pi[V_T(\sw-\omega)]<\infty$. The continuous process $\Lambda_t(\omega,\sw):=\sw_t-\omega_t$ is of integrable variation with respect to $\pi$, hence we can define its dual predictable projection $\tilde A(\omega,\sw)$ with respect to $\{\{\emptyset,\C\}\otimes \G\}^{\pi}$, the $\pi$-completion of $\{\emptyset,\C\}\otimes \G$. 
In particular, $\tilde A$ is an integrable variation process on $\C\times\C$. By Lemma~\ref{lem proj}, we may assume that $\tilde{A}$ does not depend on the first coordinate, thus it corresponds to a $\G^{\nu}$-predictable integrable variation process $A$ on $\C$, in the sense that $\tilde A(\omega,\sw)=A(\sw)$, which gives \eqref{eq: FV}. Altogether we have
\begin{eqnarray}\label{eq: dpp}\textstyle
\EE^\pi\left[\int_0^T X_t(\sw)d\Lambda_t(\omega,\sw)\right]=\EE^\nu\left[\int_0^T X_t(\sw)d A_t(\sw)\right],
\end{eqnarray}
for every $\G^{\nu}$-predictable bounded process $X$. 
Moreover, we have
\begin{equation}\label{eq var diminishes}
\EXP^\nu[ V_T(A)]=
\EXP^\pi[ V_T(\tilde A)]\leq
\EXP^\pi[ V_T(\Lambda)]<\infty. 
\end{equation}
Now, note that the jump times of $A$ are $\G^{\nu}$-predictable, by \cite[Theorem~B, page~xiii]{DellacherieMeyerB}, and that for each jump time $\tau$, $\Delta A_{\tau}=\EXP^{\pi}[\Delta \Lambda_{\tau}| \{\emptyset,\C\}\otimes \G^{\nu}_{\tau-}]=0$ a.s., from the continuity of $\Lambda$; see \cite[Theorem~VI.76]{DellacherieMeyerB}. Therefore, $A$ is continuous.

We now define the process $\xi$ as in $(i)$, and need to show that it is a $(\nu,\G^{\nu})$-Brownian motion.
For $0\leq s<t\leq T$ and $f_s\in L^\infty(\G_s^{\nu})$, we have
\begin{align*}
\EXP^\nu [(\xi_t(\sw)-\xi_s(\sw))f_s(\sw)] &\textstyle  = \EXP^\nu [(\sw_t - \sw_s - \int_{s}^{t} dA_u(\sw) )f_s(\sw)] \\
&\textstyle = \EXP^{\pi} [(\omega_t - \omega_s )f_s(\sw)] + \EXP^{\pi} [( \int_{s}^{t}  d\Lambda_u -  \int_{s}^{t}dA_u(\sw) )f_s(\sw)]\\
& \textstyle= \EXP^{\pi} [( \int_{s}^{t} d\Lambda_u -\int_{s}^{t}  dA_u(\sw) )f_s(\sw)] = 0,
\end{align*}
where the third equality follows since $\omega$, which is a $(\gamma,\F^{\gamma})$-martingale, is consequently by causality a $(\pi, \F^{\gamma}\otimes \G)$-martingale, thus also a $(\pi, (\F^{\gamma}\otimes \G)^{\pi})$-martingale and in particular a $(\pi,\F^{\gamma}\otimes \G^{\nu})$-martingale. The last equality follows from \eqref{eq: dpp} with $X_t:=f_s1_{]s,T]}(t)$. This shows that $\xi$ is a $(\nu,\G^{\nu})$-martingale, and we conclude by an application of L\'evy theorem together with Girsanov theorem; indeed, the quadratic variation of $\xi$ at $t$ must be $t$, by the assumption $\nu\ll\gamma$. 

Finally, with the aid of \eqref{eq var diminishes}, we see that the transport plan $\hat{\pi}:= (\xi,id)_{\#}\nu$ is  optimal for \eqref{eq: finiteness_gen}. Moreover, since $\sw$ is continuous, the uniqueness of the semimartingale decomposition entails that the process $A$ found in the proof of $(ii) \Rightarrow (i)$ must have not depended on which $\pi$ (with finite cost in \eqref{eq: finiteness_gen}) we started with.
\end{proof}

We now state the main theoretical result of the article, that provides a necessary and sufficient condition for a Brownian motion to remain a semimartingale in an enlarged filtration. We use the notations introduced in Section~\ref{sect: CFG}.

\begin{theorem}[Semimartingale preservation property]\label{thm: DM_new}
The following are equivalent:
\begin{itemize}
\item[(i)] any process $B$ which is a Brownian motion
in its natural filtration $\F^B$ 
on some probability space { $(\Omega,\J,\PROB)$}, remains a semimartingale in the enlarged filtration ${ \G^B}$;
\item[(ii)] the causal transport problem \eqref{eq: finiteness_gen} is finite for some measure $\nu\sim\gamma$.
\end{itemize}
Moreover, when (i)-(ii) hold, and denoting by $B=\widetilde B + N$ the semimartingale decomposition of $B$ in ${ \G^B}$,
we have that $(\widetilde B,B)$ is a causal coupling with respect to $\F^{\tilde B}$ and ${ \G^B}$.
\end{theorem}

When (i)-(ii) hold, the idea is that $(\widetilde B,B)$ is an optimal coupling (possibly under a different measure) for a causal transport problem as in \eqref{eq: finiteness_gen}.

\begin{proof}
$(ii)\Rightarrow(i)$:\ By Theorem~\ref{thm: gWnu_gen},
there exists a continuous, $\G^{\nu}$-adapted, integrable variation process $A$ such that the process $\xi_t(\sw) :=\sw_t-A_t(\sw)$ is a $(\nu,\G)$-Brownian motion. Since $\nu\sim\gamma$, Girsanov theorem implies that $\sw$ is a $(\gamma,\G)$-semimartingale.
Moreover, since $\sw$ is the coordinate process on $\C$, and from $B_\#\PP=\gamma$ and ${ \G^B}=B^{-1}(\G)$, we have that $B$ is a $(\PP,{ \G^B})$-semimartingale.

$(i)\Rightarrow(ii)$:\ Let $B=M+U$ be the semimartingale decomposition of $B$ in ${\G^B}$, with $M$ a $(\PP,{ \G^B})$-Brownian motion and $U$ a finite variation process, so that in particular $V_T(U)<\infty$.
Since $0<(1+V_T(U))^{-1}\leq 1$, we have
\[\textstyle
c^{-1}:=\EE^\PP\left[(1+V_T(U))^{-1}\right]\in(0,1],
\]
and $Z_T:=c\,(1+V_T(U))^{-1}>0$ $\PP$-a.s., with $\EE^\PP[Z_T]=1$. We can then define a probability measure $\Q$ on $(\Omega,\F^B_T)$ via $\frac{d\Q}{d\PP}:=Z_T$, so that $\Q\sim\PP$ and
\begin{equation}\label{eq VU}\textstyle
\EE^\Q[V_T(U)]=c\EE^\PP\left[\frac{V_T(U)}{1+V_T(U)}\right]\leq c<\infty.
\end{equation}

Let $Z$ be the $(\PP,{ \G^B})$-martingale defined from $Z_T$. Then, by Girsanov theorem, $B$ has decomposition
\[\textstyle
B_t=\widetilde M_t+
\left(\int_0^t\frac{d\langle Z,M\rangle_s}{Z_s}+U_t\right),
\]
where $\widetilde M:=M-\int_0^.\frac{d\langle Z,M\rangle_s}{Z_s}$ is a $(\Q,{ \G^B})$-Brownian motion.
Moreover,
\begin{eqnarray*}\textstyle
\EE^\Q\left[\int_0^T\left|\frac{d\langle Z,M\rangle_s}{Z_s}\right| \right] 
&=& \textstyle\EE^\PP\left[
\int_0^T\left|d\langle Z,M\rangle_s\right| \right]
\leq \EE^\PP\left[\sqrt{\langle Z,Z\rangle_T}\sqrt{\langle M,M\rangle_T}\right]\\
&=& \textstyle \sqrt{T}\EE^\PP\left[\sqrt{\langle Z,Z\rangle_T}\right]\leq K\sqrt{T}\EE^\PP\left[\sup_{t\in[0,T]}Z_t\right]
\leq cK\sqrt{T},
\end{eqnarray*}
by the Kunita-Watanabe inequality and the Burkholder-Davis-Gundy inequality (with constant $K$). Together with~\eqref{eq VU}, this implies that the process $N:=\int_0^.\frac{d\langle Z,M\rangle_s}{Z_s}+U$ is of $\Q$-integrable variation. This shows that $(ii)$ holds with $\nu:=B_\#\Q$ ($\gamma=B_\#\PP$ and $\Q\sim\PP$ imply $\nu\sim\gamma$), { as $\pi=(\widetilde M,B)_\#\Q\in\Pi^{\F,\G}(\gamma,\nu)$ has finite cost in \eqref{eq: finiteness_gen}.}

Finally, by Lemma~\ref{lem: causal wiener equivalence}, $(\widetilde B,B)$ is a causal coupling  with respect to $\F^{\tilde B}$ and ${ \G^B}$.
\end{proof}

From the above results it is clear that if there is one causal transport in $\Pi^{\F,\G}(\gamma,\nu)$, for some measure $\nu\sim\gamma$, for which the difference of the coordinates is of integrable variation and a.s.\ absolutely continuous, then the finite variation part in the semimartingale decomposition of the Brownian motion in the enlarged filtration is also absolutely continuous, see Section~\ref{sec abs}.

\begin{remark}[Multidimensional processes]\label{rem multi}
We want to point out that the previous theorems can be easily extended to a multidimensional setting.
Indeed, instead of the path space $\C$, that accommodates $1$-dimensional continuous processes, we can consider $\C^d$, path space for $d$-dimensional continuous processes. In this case, we write $\{(a^i_t)_{t\in [0,T]}\}_{i=1}^d\mapsto V_T(a):=\sum_{i=1}^d V_T( a^i)$ in \eqref{eq: FV}-\eqref{eq: finiteness_gen} for the variation of a multidimensional process. Then the proof of Theorem~\ref{thm: gWnu_gen} follows exactly the same arguments, where the dual projections are now taken componentwise. As for Theorem~\ref{thm: DM_new}, one should define $Z_T$ as $c/(1+\sum_{i=1}^dV_T(U_i))$ instead.
\end{remark}

\subsection{The absolutely continuous case}\label{sec abs}

In many well-known filtration enlargements, the finite variation part in the semimartingale decomposition of the Brownian motion in the enlarged filtration is absolutely continuous with respect to Lebesgue (as in the examples \eqref{eq:atoms}, \eqref{eq:Brownian bridge} and \eqref{eq:progressive} above), i.e.\ is of the form
$$dB_t=d\widetilde B_t+b_t(B)dt.$$
This is true, for example, in the case of initial enlargement of filtrations under Jacod's assumption (see \cite{Jac85}) and under Yor's method (see \cite[Sect.~12.1]{Yor97}), as well as in the case of progressive enlargement with a random time (see \cite{JY78} and \cite{Jeu80}); for general enlargements see \cite{ADI}. That is why this is a framework of major interest which deserves a deeper analysis. In analogy to Theorems~\ref{thm: gWnu_gen} and~\ref{thm: DM_new}, we can give necessary and sufficient conditions for such a decomposition to hold, together with a characterization of $b$ in terms of causal transport.

Let us introduce a convenient notation. For a function $h\in\C$, we denote 
$$
\abs{h_t} := \left \{
\begin{array}{ll}
\dot{h}_t &,\mbox{ if $h$ is absolutely continuous}\\
+\infty &,\mbox{ else}.
\end{array}
\right .
$$

\begin{theorem}\label{thm: gWnu}
Let $\nu$ be some measure on $\C$ such that $\nu\ll\gamma$, and let $\rho:\R\to \R_+$ be a convex even function such that $\rho(+\infty)=+\infty$ and $\rho(0)=0$. Then the following are equivalent:
\begin{itemize}
\item[(i)] for some $\G$-predictable process $\alpha=\alpha(\sw)$ such that 
$$\textstyle \EE^{\nu}\left [\int_0^T\rho(\alpha_s)ds\right ]< \infty,$$  the process $\xi_t(\sw) :=\sw_t-\int_0^t\alpha_s(\sw)ds$ is a $(\nu,\G^\nu)$-Brownian motion;
\item[(ii)]  the following causal optimal transport problem is finite: 
\begin{equation}\label{eq: finiteness} \textstyle
\inf\limits_{\pi\in\Pi^{\F ,\G }(\gamma, \nu)}\EE^\pi\left[\int_0^T\rho(\abs{\sw_t-\omega_t})dt\right].
\end{equation}
\end{itemize}
Moreover, whenever (i)-(ii) hold, then $\hat{\pi}:= (\xi,id)_{\#}\nu $ belongs to $\Pi^{\F ,\G }(\gamma, \nu)$, it is optimal for \eqref{eq: finiteness}, and for every $\pi\in\Pi^{\F ,\G }(\gamma, \nu)$ with finite cost in \eqref{eq: finiteness}, it holds that the process $\tilde \alpha_t(\omega,\sw):=\alpha(\sw)$ 
equals the predictable projection of $\abs{\sw_t-\omega_t}$ with respect to $(\pi,\{\emptyset,\C\}\otimes \G^{\nu})$. 
\end{theorem}

\begin{theorem}[Semimartingale preservation property]\label{thm: DM_new_2}
The following are equivalent:
\begin{itemize}
\item[(i)] any process $B$ which is a Brownian motion
in its natural filtration $\F^B$ 
on some probability space { $(\Omega,\J, \PROB)$}, remains a semimartingale in the enlarged filtration ${\G^B}$, with decomposition
\begin{eqnarray}\label{eq: BBa}
dB_t=d\widetilde B_t + b_t(B)dt;
\end{eqnarray}
\item[(ii)] the causal transport problem \eqref{eq: finiteness} is finite for some measure $\nu\sim\gamma$ and some function $\rho$ as in Theorem~\ref{thm: gWnu}.
\end{itemize}
Moreover, if $(ii)$ holds for $\nu=\gamma$, then the value of the causal transport problem \eqref{eq: finiteness} equals $\EE^\gamma[\int_0^T\rho(b_t(B))dt]<\infty$, hence the information drift in \eqref{eq: BBa} is $\rho$-integrable.
\end{theorem}

\begin{remark}\label{rem CM}
For $\rho(x)=x^2/2$, the cost in \eqref{eq: finiteness} is called Cameron-Martin cost.
In this case, finiteness of problem \eqref{eq: finiteness} for $\nu=\gamma$ is equivalent to square integrability of the drift in \eqref{eq: BBa}, by Theorem~\ref{thm: DM_new_2}. When this holds, one can apply Girsanov theorem, which ensures that $B$ is a Brownian motion {in ${ \G^B}$} under a change of measure. Therefore, by martingale representation, the $H'$-hypothesis between $\F^B$ and ${ \G^B}$ follows, i.e.\ all $\F^B$-semimartingales remain semimartingales w.r.t. ${ \G^B}$. Square integrability of the drift holds for example when initially enlarging with a discrete random variable as in case (1) of Section \ref{ex:BB} { when the variable takes finitely many values}, while it fails in the Brownian bridge case and for progressive enlargements with last hitting times, as in (2) and (3) of Section \ref{ex:BB}. We stress the fact that the semimartingale property  being preserved by the Brownian motion is usually not enough to guarantee the $H'$-hypothesis, see \cite{JY79}. 
In the case of initial enlargements with a random variable, we also have that the value of the causal problem \eqref{eq: finiteness} equals the mutual information between the Brownian motion $B$ and such random variable, see Section~\ref{sect:con}.
\end{remark}

The proofs of the above theorems follow the same steps of the proofs of Theorems~\ref{thm: gWnu_gen} and~\ref{thm: DM_new}, so we omit them. One simply observes that $\rho(x)\geq m|x|+n$ for some $m>0$, so the relevant processes in the proofs are of integrable and a.s.\ absolutely continuous variation. Then one recalls that for an integrable variation process $X$ which is absolutely continuous, say $X=\int_0^.a_tdt$, the dual predictable projection of $X$ w.r.t.\ some filtration $\H$ is also absolutely continuous, and is indistinguishable from both $\int_0^.{}^pa_tdt$ and $\int_0^.{}^oa_tdt$, where $^pa$ and $^oa$ are the predictable and optional projections of $a$ w.r.t. $\H$, respectively. 
\begin{remark}[General continuous semimartingales]\label{rem semi}
The theorems above, as well as those in the non-absolutely continuous setting, have an analogue outside the Brownian framework. In order to establish this, we need a condition for transport plans that generalizes the concept of causality introduced in Definition~\ref{def:causality}, namely:
\begin{equation}\label{eq:gen_caus}
\EE^\pi[(\omega_t-\omega_s)f_s(\sw)]=0\qquad \forall\ 0\leq s<t\leq T,\, f_s\in L^\infty(\C,\G_s^\nu,\nu).
\end{equation}
In particular, if $X$ is a continuous semimartingale on a probability space { $(\Omega,\F^X,\PROB)$}, which remains a semimartingale in the enlarged filtration ${ \G^X}$ with canonical decomposition
$X=\widetilde X + N$, then the transport plan $(\widetilde X,X)_\#\PROB$ satisfies \eqref{eq:gen_caus}.

In this framework, an analogue of Theorem~\ref{thm: gWnu_gen} can be established, where now the process $\xi$ in (i) is only required to be a $(\nu,\G^{\nu})$-martingale, and where the optimal transport problem in (ii) is formulated over transport plans in $\Pi(\mu, \nu)$, for some martingale law $\mu$, which satisfy \eqref{eq:gen_caus}. This result then leads to an analogue of Theorem~\ref{thm: DM_new}, giving a necessary and sufficient condition for any continuous semimartingale $X$ to remain a semimartingale in the enlarged filtration ${ \G^X}$.
In the same way one has the analogues of Theorems~\ref{thm: gWnu} and~\ref{thm: DM_new_2} for general continuous semimartingales.
\end{remark}

Let us go back to the absolutely continuous Brownian setting. The proofs sketched before the remark above, are of stochastic analysis flavour, exactly as for Theorems~\ref{thm: gWnu_gen} and~\ref{thm: DM_new}. We now describe what the optimal transport perspective has to say in the absolutely continuous case. An interesting observation is that in this setting we can actually say more about the problem dual to \eqref{eq: finiteness}, which, as Theorem~\ref{prop duality} below points out, corresponds to
\begin{align} \textstyle
\sup\limits_{\substack{  \psi\in C_b(\C),\, h\in \fH \\\psi(\sw)\leq \int_0^T \rho(\abs{\sw_t-\omega_t})dt +h(\omega,\sw)  }} \label{eq dual wiener}
\left\{\int \psi(\sw)\nu(d\sw) \right \} ,
\end{align}
where 
\begin{equation*}
\textstyle \fH\,\,:=\,\, span\left (\left\{ g(\sw)\left[f(\omega)- \EXP^\gamma[ f| \F_t](\omega)\right]\, : \,f\in C_b(\C),g\in B_b(\C,\G_{t}),t\in [0,T]\right \} \right ).\label{eq h wiener}
\end{equation*} 
The next result is proved in Appendix \ref{sec Orlicz proof}, where its ingredients are more closely examined. The necessary elementary facts on Orlicz spaces are given in Appendix~\ref{Sec orlicz}. We stress that even without anticipation of information (i.e.\ $\F=\G$) this is a novel result. 

Let us define the \emph{refined dual} problem as
\begin{equation} \textstyle
\sup\left\{\int\left[\int_0^T F_t(\sw)d\sw_t - \int_0^T \rho^*(F_t(\sw))dt\right ] \nu(d\sw) :\,\, F\in S_a(\G)\right\},\label{eq: dual easy}
\end{equation}
where $\rho^*$ denotes the convex conjugate of $\rho$, and $S_a(\G)$ denotes the set of simple previsible processes w.r.t. $\G$, namely,
$$\textstyle S_a(\G):=\left\{\sum_{i=1}^m F^i(\sw)\IND_{(\tau_i,\tau_{i+1}]}(t):m\in\mathbb{N},0\leq\tau_1\leq\dots\leq \tau_m\leq T\mbox{ $\G$-stopping times},F^i\in B_b(\G_{\tau_i})\right\},$$
so the first integral in \eqref{eq: dual easy} is defined as a finite sum as customary.

\begin{theorem}\label{thm orlicz}
Let $\nu, \rho$ be as in Theorem~\ref{thm: gWnu}. Suppose further that $\rho$ is strictly convex and satisfies for some $C,x_0>0$ and $\ell>1$:
$$\textstyle \rho(x)=0\Leftrightarrow x=0\,\,,\,\, \frac{\rho(0)}{0}=0\,\,,\,\, \frac{\rho(+\infty)}{+\infty}=+\infty\,\,,\,\, \rho(2x)\leq C\rho(x) \mbox{ and $\rho(x)\leq \frac{\rho(\ell x)}{2\ell}$ if $x\geq x_0$ }.$$ 
 Then:\\
  (i)\,\, The primal \eqref{eq: finiteness}, the dual \eqref{eq dual wiener}  and the refined dual \eqref{eq: dual easy}, have the same value.
  
  \vspace{5pt}
   \noindent  From now on we assume that this common value is finite. 
   \vspace{5pt}
   
\noindent (ii)\,\,
 The refined dual \eqref{eq: dual easy} can be computed (without changing its value) over $M^{\rho^*}$, the closure of $S_a(\G)$ w.r.t.\ the so-called gauge norm 
 \begin{equation}
 \textstyle F\mapsto\|F\|_{\rho^*}:=\inf \{\beta >0: \int \int_0^T \rho^*(F_t(\sw)/\beta)dt\,\nu(d\sw)\leq 1\}, \label{eq gauge}
 \end{equation} 
 and it is attained there by a unique optimizer $\hat{F}=\hat{F}(\sw)$.\\
\noindent
(iii)\,\, The optimal drift $\alpha$ in Theorem~\ref{thm: gWnu}(i) is related to $\hat{F}$ through 
\begin{equation}\label{eq rho rho star}
\rho^*(\hat{F}_t(\sw))+ \rho(\alpha_t(\sw))=\alpha_t(\sw)\hat{F}_t(\sw) \,\,\,\ d\nu\times dt\mbox{-a.s}, 
\end{equation}
namely $\alpha_t(\sw)= (\rho^*)'(\hat{F}_t(\sw))$,
or equivalently, $\hat{F}_t(\sw) \in \partial \rho(\alpha_t(\sw))$, {with $\partial$ denoting sub-differential.}   \\
(iv)\,\, $\xi_t(\sw) :=\sw_t-\int_0^t\alpha_s(\sw)ds$ is a $(\nu,\G^\nu)$-Brownian motion, so if further $\gamma\ll\nu$ we have that the canonical process $\sw$ is a $(\gamma,\G^\nu)$-semimartingale.
\end{theorem}

The typical examples for which the given conditions on $\rho$ are satisfied, are power functions { $\rho(x)\sim |x|^p$} with exponent $1<p<+\infty$, which covers the Cameron-Martin case, as well as $\rho(x)=|x|^a(1+|\log|x||)$ for $a>1$; see the comments after \cite[Ch. II.2.3, Corollary 4]{RaoRen}. 

\begin{remark}
We stress that the dual \eqref{eq dual wiener} is most often not attained on continuous functions. Still, the refined dual \eqref{eq: dual easy} admits an optimizer which induces a formal optimal element for \eqref{eq dual wiener} by setting $h(\omega,\sw):= \int_0^T \hat{F}_t(\sw)d\omega_t$ and
$$\textstyle \psi(\sw) := \int_0^T \hat{F}_t(\sw)d\sw_t - \int_0^T \rho^*(\hat{F}_t(\sw))dt = \int_0^T \hat{F}_t(\sw)d\xi_t(\sw) + \int_0^T \rho(\alpha_t(\sw))dt.$$
Notice that, even though the optimal $\hat{F}$ depends on the cost function $\rho$, the optimal drift $\alpha$ does not. {In the Cameron-Martin case $\rho(x)=x^2/2$, we actually get from \eqref{eq rho rho star} that $\hat{F}=\alpha$ and so $\psi(\sw)= \int_0^T \alpha_t(\sw)d\sw_t - \int_0^T (\alpha_t(\sw))^2/2\,dt$. Furthermore, in the absence of enlargement (i.e.\ $\F=\G$) 
 we find by Girsanov theorem the identity 
\begin{equation*}\textstyle
 \frac{d\nu}{d\gamma}(\sw)=\exp(\psi(\sw)).
\end{equation*} 
In words: the causal Kantorovich potential $\psi$ between Wiener measure and $\nu$ is the logarithm of their relative density.}
\end{remark}

In the proof of Theorem~\ref{thm orlicz} (see Appendix \ref{sec Orlicz proof}) we extend the stochastic integral $\int_0^T F_t(\sw)d\sw_t$ beyond simple $\G$-previsible integrands via functional analytic arguments, much inspired by \cite{Leo_Girsanov}. Of course,  this could have been done via Theorem~\ref{thm: gWnu}, using that a fortiori the coordinate process $\sw$ is a $(\nu,\G)$-semimartingale. We avoided this to show that there is a true transport/functional method for this. Likewise, Point $(iv)$ is obtained without using previous results. 

\begin{remark}[Multidimensional processes]\label{rem multi abs}
As seen in Remark~\ref{rem multi} for the general case, also the theorems of this section have an analogue in the multidimensional setting. It suffices to define the gauge norm \eqref{eq gauge} as acting on the euclidean norm of $F_t(\sw)$, interpret the r.h.s.\ of \eqref{eq rho rho star} as inner product, etc. This is straightforward, so we do not give the details.
\end{remark}

\subsection{Initial enlargement: Jacod's condition, entropy, and mutual information}\label{sect:con}
As explained in the ``Comparison with Jacod's condition'' Section in \cite{ADIGirsanov}, Jacod's method for initial enlargements \cite{Jac85} (see our Section \ref{sect: CFG}) can be interpreted in the following way: starting with a Brownian motion $B$ on the probability space { $(\Omega,\F^B,\PROB)$}, considering an initial enlargement of the Brownian filtration $\F^B$ with a random variable $L(B)$ and assuming that for almost all $x$, ${\PP^x}:=\PP(\cdot|L(B)=x)\ll \PP$, one applies Girsanov theorem and finds $B-A^x$ to be a (local) martingale w.r.t.\ ${\PP^x}$. Then,  combining these, one obtains that $B-A^{L(B)}$ is a (local) martingale w.r.t.\ $\PP$ and the enlarged filtration ${ \G^B=\F^B\vee\sigma(L(B))}$. 
Remember that under Jacod's condition the finite variation process $A^{L(B)}$ is absolutely continuous with respect to Lebesgue, that is, $A_t^{L(B)}=\int_0^t\alpha_sds$. 
Notably, there is a causal optimal transport counterpart to the method just described. Denote  $$\ell(dx)\,:=\,L(B)_\#\PROB, \,\,\,\mbox{   and    $\gamma^{L=x}\, =\,$ ``conditional law of $B$ given $L(B)=x$''} .$$ 

\begin{lemma} { Set $\G=\F\vee\sigma(L)$.} We have
\begin{equation}\label{eq:dis} 
\inf\limits_{\pi\in{ \Pi^{\F ,\G}}(\gamma, \gamma)}\EE^\pi\Big[\int_0^T
\rho(\abs{\sw_t-\omega_t})dt\Big] = \int \Big\{ 
\inf\limits_{\pi\in\Pi^{\F,\F}(\gamma, \gamma^{L=x})}\EE^\pi\Big[\int_0^T
\rho(\abs{\sw_t-\omega_t})dt\Big] 
 \Big\}\ell(dx).
\end{equation}
\end{lemma}

Observe that the integrand in the r.h.s.\ of \eqref{eq:dis} is a causal transport problem in itself, but without enlargement of filtration. The proof relies on easily checking that for $\pi\in{ \Pi^{\F ,\G}}(\gamma, \gamma)$ one has $\pi^{L(\sw)=x}\in\Pi^{\F ,\F }(\gamma, \gamma^{L=x})$, and ultimately on a standard measurable selection argument, and so we omit it. 

In the Cameron-Martin case of $\rho(x)=x^2/2$, \cite[Lemma 5]{Lassalle2} implies that the integrand in the r.h.s.\ of \eqref{eq:dis} equals the relative entropy of $\gamma^{L=x}$ w.r.t. $\gamma$, whenever this is finite.  For us this means that
\begin{equation}\label{eq:entr} \textstyle
\inf\limits_{\pi\in{ \Pi^{\F ,\G}}(\gamma, \gamma)}\EE^\pi\Big[\frac12\int_0^T
(\abs{\sw_t-\omega_t})^2dt\Big] = \int \mathrm{Ent}(\gamma^{L=x}|\gamma) \ell(dx).
\end{equation}
Since the relative entropy $\mathrm{Ent}(\gamma^{L=x}|\gamma)$ is further integrated w.r.t.\ the law of $L(B)$, we get that the r.h.s.\ in \eqref{eq:entr} corresponds to the so-called \emph{Mutual Information} between $B$ and $L(B)$, denoted by $\mbox{I}(B,L(B))$.
It is defined as the relative entropy of the joint law $P_{B,L(B)}$ w.r.t.\ the decoupling measure $P_B\otimes P_{L(B)}$, namely: 
$$\textstyle \mbox{I}(B,L(B)):=\mathrm{Ent}(P_{B,L(B)}|P_B\otimes P_{L(B)})= \int \mathrm{Ent}(\gamma^{L=x}|\gamma) \ell(dx).$$
On the other hand, by Theorem~\ref{thm: DM_new_2}, the l.h.s.\ in \eqref{eq:entr} is finite if and only if the information drift $\alpha$ in the semimartingale decomposition of $B$ w.r.t.\ ${ \G^B}$ is square integrable, in which case the value of the causal transport problem equals $\frac12\EE^\gamma[\int_0^T\alpha_t^2dt]$. In \cite{PikovKaratz} (see also \cite{ADI}, \cite{AIS}) it is proved that this value corresponds to the additional utility obtained by an investors who maximizes the expected log-utility of terminal wealth in a certain complete market model w.r.t.\ ${ \G^B}$, compared to an investor w.r.t.\ $\F^B$. Further, it is known that this value also equals the relative entropy $\mathrm{Ent}(\PP|\Q)$, where  $\Q$ is a probability measure under which $B$ is a $(\Q,{ \G^B})$-Brownian motion; see Remark~\ref{rem CM}. Putting things together, we have 
\begin{corollary}
Assuming that $\gamma^{L=x}\ll\gamma ~\forall x$ $\ell$-a.s., then
\[\textstyle \inf\limits_{\pi\in{ \Pi^{\F ,\G}}(\gamma, \gamma)}\EE^\pi\Big[\frac12\int_0^T
(\abs{\sw_t-\omega_t})^2dt\Big] =
\mathrm{Ent}(\PP|\Q) = \int \mathrm{Ent}(\gamma^{L=x}|\gamma) l(dx)=\mbox{I}(B,L(B)).
\]
\end{corollary}
Note that the equality with the mutual information recovers the result of \cite[Theorem 5.13]{ADI} using our methods.
If the initial enlargement is done by a discrete random variable as in Section \ref{ex:BB}-(1), then
\[
\textstyle \inf\limits_{\pi\in{ \Pi^{\F ,\G}}(\gamma, \gamma)}\EE^\pi\Big[\frac12\int_0^T
(\abs{\sw_t-\omega_t})^2dt\Big] =
{ -\sum\limits_np_n\ln(p_n)},
\]
where {$p_n=\PP(C_n)$}, and the term on the r.h.s.\ is referred to as the entropy of the partition { $\{C_n\}_{n\in\NN}$};  see \cite{Yor85} and \cite{AIS}.

\section{Robust transport bounds for stochastic optimization}\label{sect:app}

In this section we show how the causal transport framework allows us to give robust estimates for the value of additional information, as well as model sensitivity, for some classical stochastic optimization problems in continuous-time. By value of information we mean the difference between the optimal value of these problems with and without additional information (i.e.\ w.r.t.\ the enlarged and the original filtration, respectively). By model sensitivity we mean the difference between the optimal value of these problems under two different probabilistic models (i.e.\ reference probability measures). For the value of information, the main idea is to take ``causal projections'' of candidate optimizers in the problem with the larger filtration, so building a feasible element in the problem with the smaller filtration, and making a comparison possible. For model sensitivity, it means to project an optimizer under one model in order to build a feasible element for the other model, which enables a direct comparison.  In discrete-time and in the setting of model sensitivity in multistage stochastic programming, this idea goes back  to \cite{Pflug,PflugPichler}.

We start with Section~\ref{sec opt stop}, on optimal stopping problems, for which the outlined projection approach is more delicate and fully novel to the best of our knowledge (even in discrete-time). 
Then in Section~\ref{sect: utmax} we deal with utility maximization with portfolio constraints; this is a prominent example of a controlled linear system, and indeed the same arguments would be applicable to such systems in general. 
In both optimization problems considered below, we will obtain robust estimates in terms of causal minimization problem.  Finally in Section~\ref{sec project causal} we provide a link between causal transport and projections of processes which is needed for Section~\ref{sec opt stop}, and is very illustrative in its own. 

\subsection{Optimal stopping}\label{sec opt stop}
Here we consider the framework of Section~\ref{sect: CFG}, with canonical space $\C:=\C[0,T]$ and filtrations {$\F\subseteq\G$} on it. 
We begin with the definition of a randomized stopping time:

\begin{definition}
A randomized stopping time $\Sigma$ with respect to a filtration $\H$  and a probability measure $\mu\in\P(\C)$, written $\Sigma \in RST(\H,\mu)$, is an  increasing right-continuous $\H$-adapted process on $[0,T]$, with $\Sigma_0=0$ and $\Sigma_T=1,\, \mu-a.s.$
\end{definition}
This notion generalizes the concept of stopping time, say $\tau$, according to which a path $\omega$ is stopped at a unique point in time $\tau(\omega)$. Stopping according to a randomized stopping time $\Sigma$ means that a path $\omega$ is stopped in $[0,t]$ with probability $\Sigma_t(\omega)$. We recommend \cite[Sect.~3.2]{BCH} for a modern view on this matter, and refer to \cite{BaxterChacon} for the original motivation/definition.

The next lemma is of fundamental importance for our applications. It identifies what causal dual optional projections do to randomized stopping times:

\begin{lemma}\label{lem crucial stopping} 
Let $\Sigma\in RST(\G,\nu)$. Then, for any $\mu\in\P(\C)$ and any causal transport plan $\pi\in\Pi^{\F,\G}(\mu,\nu)$, there is a randomized stopping time $\tilde{\Sigma}\in RST(\F,\mu)$ such that
\begin{align*}
\textstyle
\EE^\pi \left [ \int_0^T \ell(\omega,t)d\Sigma_t(\sw)\right ] = \EE^\mu \left [\int_0^T \ell(\omega,t)d\tilde{\Sigma}_t(\omega)\right ],
\end{align*}
for all $\F$-optional processes $(\omega,t)\mapsto \ell(\omega,t)$ which are bounded or positive.
\end{lemma}

\begin{proof}
Let $\tilde{\Sigma}(\omega,\sw)$ be the dual optional projection of $\Sigma=\Sigma(\sw)$ with respect to $(\pi,(\F\otimes \{\emptyset,\C\})^{\pi})$. From Lemma~\ref{lem proj}, we may assume that $\tilde{\Sigma}(\omega,\sw)= \tilde{\Sigma}(\omega)$, and from Proposition~\ref{prop crucial} below we have that $\tilde{\Sigma}$ equals the optional projection of $\Sigma$ with respect to $(\pi,(\F\otimes \{\emptyset,\C\})^{\pi})$. Moreover, by \cite[Lemma~7, Appendix~I]{DellacherieMeyerB}, we can assume $\tilde \Sigma$ to be $(\F\otimes\{\emptyset,\C\})$-optional.
This implies that $\tilde{\Sigma}$ lies in the interval $[0,1]$ too, and hence belongs to $RST(\F,\mu)$.
\end{proof}

Our purpose is to quantitatively gauge, via causal transport arguments, the dependence of optimal stopping problems on $[0,T]$ with respect to the filtration or the reference probability measure. See \cite{LambertonPages,CoquetToldo} or the seminal but unpublished work \cite{Aldousweak}, for the related issue of (qualitative) stability of these problems. Lemma~\ref{lem crucial stopping} above suggests that we should rather define optimal stopping over randomized stopping times. It is well-known that, in the non-anticipative case, one can move between formulations over stopping times and over randomized stopping times. That this is also true in the anticipative case is somewhat hidden in the aforementioned articles, so we sketch the arguments for convenience of the reader:

\begin{lemma}
\label{lem harmless}
Let $\nu\in \P(\C)$, and let $(\omega,t)\mapsto \ell(\omega,t)$ be measurable, $\F$-optional, bounded or positive. Then
\begin{align}\label{eq harmless}\textstyle
\inf\left\{ \EXP^\nu\left[\ell_\tau \right]:\, \, \tau \mbox{ a $\G$-stopping time on $\C$},\,\tau\leq T \right \} \,\, = \,\,\inf\limits_{\Sigma \in RST(\G,\nu)} \EXP^\nu\left[\int \ell_t d\Sigma_t\right].
\end{align}
Furthermore, let $(\Omega,\H,\PROB)$ be a complete filtered probability space, $X:\Omega\to \C$ measurable and $\H$-adapted with $X_\#\PROB=\nu$. Assuming that \begin{align} \forall t\leq T, \,X^{-1}(\G_T) \mbox{ is conditionally independent of $\H_t$ given $X^{-1}(\G_t)$},\label{line condition}
\end{align}
we further have that the common value in \eqref{eq harmless} equals
\begin{align}\label{eq weak formu}
\inf\left\{ \EXP^\PROB\left[\ell(X,\tau)\right]:\, \, \tau \mbox{ a $\H\vee X^{-1}(\G)$-stopping time on $\Omega$},\,\tau\leq T \right \} .
\end{align}
\end{lemma}

\begin{proof}
We first prove \eqref{eq harmless}, following \cite[Proof of Lemma 9]{CoquetToldo}. Evidently the r.h.s.\ in \eqref{eq harmless} is the lesser one. For the converse inequality, take $\Sigma\in RST(\G,\nu)$ and define $$(\omega,x)\in\C\times [0,1]\mapsto\tau(\omega,x):= \inf\{t\in [0,T]: \Sigma_t> x\},$$ 
so by \cite[Ch.\ VI.55]{DellacherieMeyerB} we have $$\textstyle \EXP^\nu\left[\int \ell_t d\Sigma_t\right] = \int\int_0^T \ell(\omega,t)d\Sigma_t(\omega)\nu(d\omega) = \int \int_0^1 \ell(\omega,\tau(\omega,x))\,dx\,\nu(d\omega).$$
Observe that for $x$ fixed and each $t$ we have $ \{\omega \in \C:\,\tau(\omega,x)>t\} =\{\omega\in\C:\, \Sigma_t\leq x \}\in \G_t$, hence $\omega\mapsto \tau(\omega,x)$ is a $\G$-stopping time on $\C$. 
Applying Fubini-Tonelli theorem, we find
$$\textstyle\EXP^\nu\left[\int \ell_t d\Sigma_t\right] =  \int_0^1\int \ell(\omega,\tau(\omega,x))\,\nu(d\omega)\,dx ,$$ and  since the integrand in the r.h.s.\ here is for each $x$
larger than the l.h.s.\ of \eqref{eq harmless}, this establishes the equality. As for \eqref{eq weak formu}, one follows the arguments in \cite[Proposition 3.5]{LambertonPages}, or more precisely their extension in \cite[Lemma 17]{CoquetToldo}.
\end{proof}

If $\nu$ above is Markov (resp.\ Wiener) and $\F=\G$, then \eqref{line condition} is equivalent to $X$ being Markov (resp.\ Brownian motion) w.r.t. $\H$. This should convey the message that  both Condition \eqref{line condition} and Problem \eqref{eq weak formu} are natural in our more general context.

We now look at optimal stopping under $(\F,\mu)$, which by the previous lemma equals
\begin{align}
\label{OF}\textstyle
\tag{$OS(\F,\mu)$}
v^{\F,\mu}:=\inf\limits_{\Sigma \in RST(\F,\mu)} \EXP^\mu\left[\int \ell_td\Sigma_t\right].
\end{align}
We want to compare this problem with the one where extra information (anticipation) is available and/or the law of the process to be stopped is different, namely (again by Lemma \ref{lem harmless})
\begin{align}
\label{OG}\textstyle
\tag{$OS(\G,\nu)$}
v^{\G,\nu}:=\inf\limits_{\Sigma \in RST(\G,\nu)} \EXP^\nu\left[\int \ell_t d\Sigma_t\right].
\end{align}
The comparison of $v^{\F,\mu}$ with $v^{\G,\mu}$ corresponds to assessing the cost of information/anticipation. On the other hand, the comparison of $v^{\F,\mu}$ with $v^{\F,\nu}$ corresponds to the study of the dependence of non-anticipating optimal stopping with respect to different reference measures, or equivalently, with respect to different processes; in other words, model sensitivity. We have:

\begin{proposition}\label{prop:OS}
Assume that $v^{\F,\mu}$ and $v^{\G,\nu}$ are both finite, and that the cost function $\ell:\C\times\R_+$ is optional and $K$-Lipschitz in its first argument with respect to a metric $d$ on $\C$, uniformly in time (i.e.\ in the second argument). Then we have
\begin{align}
\label{upper bound ST}\textstyle
v^{\F,\mu}-v^{\G,\nu}\leq K\inf\limits_{\pi\in\Pi^{\F,\G}(\mu,\nu)}\EE^\pi[d(\omega,\sw)] .
\end{align}
In particular, in the two special cases of interest we have:
\begin{enumerate}[(i)]
\label{i} \item If $\mu=\nu$, then 
\begin{align}
0\leq v^{\F,\mu}-v^{\G,\mu} \leq K\inf\limits_{\pi\in\Pi^{\F,\G}(\mu,\mu)}\EE^\pi[d(\omega,\sw)];
\end{align}
\item If $\F = \G$, then
\begin{align}\label{bound bicausal}
\textstyle
\left| v^{\F,\mu}-v^{\F,\nu} \right| \leq K\inf\limits_{\pi\in\Pi_{bc}^{\F,\F}(\mu,\nu)}\EE^\pi[d(\omega,\sw)] ,
\end{align}
where the constraint in the transport problem in the right-hand side of \eqref{bound bicausal} means that both $\pi\in \Pi^{\F,\F}(\mu,\nu)$ and $\tilde{\pi}\in \Pi^{\F,\F}(\nu,\mu)$, where $\tilde{\pi}=((\omega,\sw)\mapsto (\sw,\omega))_\#\pi$.
\end{enumerate}
\end{proposition}

\begin{proof}
Take an optimizer ${\Sigma}$ for \eqref{OG} (same argument holds for an optimizing sequence). We write this in the $\sw$-variable and consider any causal transport $\pi$ between $\mu$ and $\nu$. Let $\tilde \Sigma \in RST(\F,\mu)$ be the randomized stopping time associated to $\Sigma$, as in Lemma~\ref{lem crucial stopping}. We have
\begin{align*}\textstyle
v^{\F,\mu} \leq \EXP^{\mu} \left[ \int_0^T l(\omega,t) d \tilde \Sigma_t(\omega) \right] = \EXP^{\pi} \left[ \int_0^T l(\omega,t) d  \Sigma_t(\sw) \right].
\end{align*}
Hence the difference can be bounded above as follows
\begin{align*}\textstyle
v^{\F\mu}-v^{\G,\nu}\leq \EXP^{\pi} \left[ \int_0^T [ l(\omega,t)-l(\sw,t)]  d \Sigma_t(\sw) \right]  \leq K T \EE^\pi[d(\omega,\sw)]. 
\end{align*}
Being $\pi$ a generic causal transport between the measures $\mu$ and $\nu$, we get the bound in \eqref{upper bound ST}.

In the case (i), obviously $v^{\F,\mu} \geq v^{\G,\mu}$, since only the set of feasible optimization variables changes. As for the case (ii), exchanging the roles of $\mu$ and $\nu$ we get
\begin{equation*}
\resizebox{.95\hsize}{!}{$\left| v^{\F,\mu}-v^{\F,\nu} \right| \leq K \max\left\lbrace \inf\limits_{\pi\in\Pi^{\F,\F}(\mu,\nu)}\EE^\pi[d(\omega,\sw)], \inf\limits_{\pi\in\Pi^{\F,\F}(\nu,\mu)}\EE^\pi[d(\omega,\sw)]\right\rbrace \leq K\inf\limits_{\pi\in\Pi_{bc}^{\F,\F}(\mu,\nu)}\EE^\pi[d(\omega,\sw)] $}.
\end{equation*}
The last inequality follows since the cost $d$ is symmetric (as a metric), implying that the r.h.s.\ can be computed on $\Pi_{bc}^{\F,\F}(\mu,\nu)$ or $\Pi_{bc}^{\F,\F}(\nu,\mu)$ equivalently.
\end{proof}

Replacing the Lipschitz condition in Proposition \ref{prop:OS} by uniform continuity, one obtains analogue results involving a modulus of continuity. Also observe that $\ell(x,t)=f(x_t)$ and $\ell(x,t)=f(\sup_{s\leq t}x_s)$ satisfy the assumptions of Proposition~\ref{prop:OS}, with $d(\omega,\sw)=\|\omega-\sw\|_\infty$, if $f$ is Lipschitz. In this case the bound in \eqref{upper bound ST} can be further majorized up to a multiplicative constant by
$$ \inf\limits_{\pi\in\Pi^{\F,\G}(\mu,\nu)}\EE^\pi[V_T(\sw-\omega)] .$$

{
\begin{remark}
Note that, for any choice of filtrations $\F$ and $\G$ on $\C$ (non-necessarily satisfying $\F\subseteq\G$), Lemma~\ref{lem crucial stopping} still holds true. This means that, under the assumptions of Proposition~\ref{prop:OS}, the bound in \eqref{upper bound ST} still holds, thus giving an estimate of the difference between the optimal stopping problems of agents with different information (non-necessarily one bigger than the other).
\end{remark}
}

\subsection{Utility maximization}\label{sect: utmax}
This part follows in spirit the previous section. We want to compare the optimal value of expected utility from terminal wealth, over a fixed finite time horizon $[0,T]$, when the reference filtration is enlarged by anticipation of information in the sense of Section~\ref{sect: CFG}.
{A wide literature is devoted to the utility maximization problem, in complete or incomplete markets, and with or without additional constraints; see \cite{XuShrI,KaratLehShreXu,KarLehShr,CvitKaratzas} among the earliest articles on the subject. Pikovsky and Karatzas~\cite{PikovKaratz} were the first ones to include anticipation of information. 
In a complete market, and for initial filtration enlargements, they provide the explicit value of this anticipation of information, in terms of log-utility maximization, with or without short-selling constraints; see also \cite{AIS,ADI}.

In this section we consider possibly incomplete markets, and any kind of anticipation of information (not just initial), and give an estimate of the value of information in terms of utility maximization under short-selling constraints, for a class of utility functions which includes the logarithm among other well-known ones; see Assumption~\ref{assump:utility}.
In order to do this, we set $\X = \Y = \C^d = \C([0,T],\R^d)$, which is the space of continuous $\R^d$-valued functions on the interval $[0,T]$, and keep the notation $\omega, \sw$ for the coordinate processes.}
Let $(\Omega,\PP)$ be a probability space, equipped with $\F^B$, the natural filtration of a $d$-dimensional Brownian motion $B= (B_1,\dots,B_d)^*$, i.e.\ $\F^B = \sigma(B_1,\dots,B_d)$, augmented so as to satisfy the usual conditions. We use notation analogous to that of Section~\ref{sect: CFG}, and denote by {$\G^B$ the enlargement of the filtration $\F^B$ with some anticipation of information on the evolution of $B$.
Hence $\G^B$} represents the information available to the informed agent. Throughout we assume:

\begin{assumption}\label{ass:Bum}
The process $B$ remains a semimartingale with respect to ${ \G^B}$, say with semimartingale decomposition $B=\widetilde{B}+A$.
\end{assumption}

We consider a financial market consisting of a riskless asset (bond), which we normalize to $1$, and $m\leq d$ risky assets whose price dynamics are described by the stochastic equations
\[ \textstyle
dS^i_t = S^i_t\big( b^i_t\ dt+ \sum_{j=1}^d \sigma^{i,j}_t dB^j_t\big) , \quad i = 1, \dots, m,
\]
with initial condition $S^i_0=s^i_0>0$. The vector process $b = (b_t^1, \dots, b_t^m)^*$ of mean rates of return is assumed  to be $\F^B$-progressively measurable and $L$-Lipschitz uniformly in time, i.e. 
\begin{equation}\label{eq:bLip}
\textstyle |b^i_t(\omg^1) - b^i_t(\omg^2) | \leq L \sum_{k=1}^d\,\, \sup_{0 \leq s \leq t}  |\omg^{1,k}_s - \omega^{2,k}_s| \quad \forall \, t, i \text{ and } \forall \, \omg^1,\omg^2 \in \C^d.
\end{equation}
The $m\times d$ volatility matrix $\textstyle \sigma_t = (\sigma^{i,j}_t)_{1\leq i \leq m, 1 \leq j \leq d} $  has full rank, it is $\F^B$-progressively measurable and $M$-Lipschitz uniformly in time, i.e. 
\begin{equation}\label{eq:sLip}
\textstyle   |{\sigma}^{i,j}_t(\omg^1) - {\sigma}^{i,j}_t(\omg^2) | \leq M \sum_{k=1}^d \,\,\sup_{0 \leq s \leq t}  |\omg^{1,k}_s - \omega^{2,k}_s| \quad  \forall\, t, i,j \text{ and } \forall \, \omg^1,\omg^2 \in \C^d,
\end{equation}
and there exists some constant $C$ s.t.\ 
$|\sigma^{i,j}_t|\leq C$ for each time $t$ and for any $i,j$ .
We denote by $\lambda^i_t$ the proportion of an agent's wealth invested in the $i$th stock at time $t$ ($1 \leq i \leq m$), the remaining proportion $1 - \sum_{i=1}^m \lambda^i_t$ being invested in the bond. We shall forbid short-selling of stocks and bond, which corresponds to the constraint $\lambda_i \in [0,1] $ for all $i$, and $\sum_{i=1}^m \lambda^i_t \leq 1$. We write $\lambda \in \mathbb{A}$ for this constraint; the case of arbitrary compact-convex constraints can be treated in the same way.
Let $\A({ \G^B})$ and $\A(\F^B)$ be the sets of \emph{admissible portfolios} for the agent with and without anticipative information, i.e.\ the sets of ${\G^B}$-, respectively $\F^B$-progressively measurable $\mathbb{A}$-valued processes $(\lambda_t)_{t\in[0,T]}$.
Denoting by $X^\lambda$ the wealth process corresponding to a portfolio $\lambda$ and starting from a unit of capital, we have $ \textstyle  
dX^\lambda_t=X^\lambda_t \lambda^*_t[b_t dt+ \sigma_t dB_t]$, that is,
\[ \textstyle
X^\lambda_t= \exp\left(\int_0^t(\lambda^*_s b_s -\frac12 || \sigma_s^* \lambda_s ||^2)ds+\int_0^t \lambda^*_s \sigma_s dB_s
\right).
\]
The above expression makes sense for portfolios in $\A({ \G^B})$ by Assumption~\ref{ass:Bum}. We also need:

\begin{assumption}\label{assump:utility}
The utility function $U:\R_+\to\R$ is concave, increasing, and such that, for some $K\in\R_+$, we have $F:=U\circ exp$ is K-Lipschitz, concave and increasing.
\end{assumption}
We remark that this assumption is fulfilled e.g.\ by $U$ negative power utility $U(x)=\frac{x^a}{a}$ for $a\leq0$, or logarithmic utility $U(x)=ln(x)$, or exponential utility $U(x)=-\frac{1}{a}e^{-ax}$ for $a\geq 1$. The function $F$ is 1-Lipschitz for the first two examples, and $e^{-a}$-Lipschitz for the last one.  

The utility maximization problem without anticipation of information is then given by
\begin{align*}
\label{UF}\tag{{ $U(\F)$}}\textstyle
{v^\F}=\sup\limits_{\lambda \in \A(\F^B)} \EXP[U(X^\lambda_T)].
\end{align*}
We proceed to compare this value with the following problem under anticipation of information:
\begin{align*}
\label{UG}\tag{{ $U(\G)$}}\textstyle
{ v^\G}=\sup\limits_{\lambda \in \A({\G^B})} \EXP[U(X^\lambda_T)],
\end{align*} 
hence obtaining a bound on the price of information relative to the risk-attitude encoded by $U$.

\begin{proposition}\label{prop: ut}
The difference between the value functions of informed and uninformed agents can be bounded as follows (with the convention $+\infty -\infty = 0$)
\begin{align}\label{bound utility}\textstyle
0\leq {v^\G -v^\F}
 \leq \tilde{K} \inf\limits_{\pi\in \Pi^{\F,\G}(\gamma,\gamma)} \EXP^{\pi} [ V_T(\sw - \omega)],
\end{align}
for some explicit constant $\tilde{K}$, see \eqref{eq:cnst}.
\end{proposition}
We recall that by the total variation of an $\R^d$-valued process $X$ we mean the sum of total variations of its components, i.e.\ $V_T(X) = \sum_{i=1}^d V_T(X_i)$. Thanks to the multidimensional version of Theorem~\ref{thm: DM_new} (see Remark~\ref{rem multi}), we have that if the causal problem in \eqref{bound utility} is finite, then its value equals $\EE^\gamma\left[V_T(A)\right]$.
The appeal of \eqref{bound utility} is that in principle one need not know the specific form of the process $A$.
\begin{proof} In the case ${ v^\F}=\infty$ we do not have anything to prove, hence we assume ${ v^\F}$ to be finite.
In the path space $\C^d$, the expected utility from terminal wealth for the agents with and without anticipative information is given by
\[\textstyle	\EE^\gamma\left[F\left(\int_0^T(\lambda^*_tb_t-\frac12 || \sigma_t^* \lambda_t ||^2)dt+\int_0^T \lambda^*_t \sigma_t d\omega_t,
\right)\right]
\]
where $\lambda$ is $\G$- and $\F$-progressively measurable, respectively.

We now fix a causal transport $\pi\in \Pi^{\F,\G}(\gamma,\gamma) \subset \P(C^d \times \C^d)$ for which the total variation $V_T(\sw - \omega)$ is $\pi$-a.s.\ finite, and consider \eqref{UF} to be solved in the $\omega$ variable and \eqref{UG} in the $\sw$ variable. Assume ${ v^\G} < \infty$ and that there is an optimizer $\hat\lambda(\sw) =(\hat\lambda^1(\sw), \dots,\hat\lambda^m(\sw))$ for problem \eqref{UG} (else one can argue in the same way for every element $\lambda^n$ of a sequence such that $\EE[U(X_T^{\lambda^n})]\to { v^\G}$ for $n\to\infty$).
We denote by $\widetilde\lambda =( \widetilde\lambda^1,\dots,\widetilde\lambda^m) $ its optional projection with respect to $(\pi,\F\otimes\{\emptyset,\C\})$, so that in particular 
$$\textstyle \widetilde\lambda^i_t(\omega)=
\widetilde\lambda^i_t(\omega,\sw)=\EE^\pi[\hat\lambda^i_t(\sw)|\F_t] = \EE^\pi[\hat\lambda^i_t(\sw)|\F_T], \,\, i=1\dots,m,$$ (for simplicity, here and in what follows, we use the notation $\EE^\pi[.|\F_t]$ for $\EE^\pi[.|\F_t\otimes\{\emptyset,\C\}]$). The last equality follows by causality and is crucial for the next argument. Note 
that $\widetilde\lambda\in\A(\F^B)$, which yields
\begin{eqnarray*}\textstyle
{ v^\F} &\geq&\textstyle \EE^\gamma\Big[F\Big(		\int_0^T(\widetilde\lambda^*_t(\omega)b_t(\omega)- \frac12 || \sigma_t^*(\omg)  \widetilde\lambda_t(\omega) ||^2 )dt+\int_0^T \widetilde\lambda^*_t(\omega) \sigma_t(\omg) d\omega_t\Big)\Big]\\
&\geq & \textstyle\EE^\pi\Big[F\Big( \EE^\pi\Big[
	\int_0^T(\hat\lambda^*_t(\sw)b_t(\omega)- \frac12 || \sigma_t^*(\omg)  \hat\lambda_t(\sw) ||^2 )dt+\int_0^T \hat\lambda^*_t(\sw) \sigma_t(\omg) d\omega_t  \Big| \F_T \Big ]
\Big)\Big] \\
&\geq& \textstyle\EE^\pi\Big[F\Big(	\int_0^T(\hat\lambda^*_t(\sw)b_t(\omega)-\frac12  || \sigma_t^*(\omg)  \hat\lambda_t(\sw) ||^2)dt+\int_0^T \hat\lambda^*_t(\sw) \sigma_t(\omg) d\omega_t \Big)\Big]  ,
\end{eqnarray*}
by Jensen's inequality, being $F$ concave and increasing, and by causality.
Therefore, by Lipschitz continuity of the function $F$, we have
\begin{multline*}\textstyle
0   \leq  { v^\G- v^\F} \leq \textstyle K\EE^\pi\Big[\Big|	\int_0^T \hat\lambda^*_t(\sw) \left(b_t(\sw) - b_t(\omega) \right)  dt -\frac{1}{2}\int_0^T \big( || {\sigma}_t^*(\sw)  \hat\lambda_t(\sw) ||^2 - || \sigma_t^*(\omg)  \hat\lambda_t(\sw) ||^2 \big)  dt 
\\ \textstyle+  \int_0^T \hat\lambda^*_t(\sw) \left( {\sigma}_t(\sw)   d\sw_t - \sigma_t(\omg) d\omega_t\right) 
\Big|\Big]
\end{multline*}
\begin{multline*}
 \leq \textstyle	K\EE^\pi\Big[\int_0^T|  \hat\lambda^*_t(\sw) ( b_t(\sw) - b_t(\omega))|  dt  +  \Big|  \int_0^T  \hat\lambda^*_t(\sw)  \left( {\sigma}_t(\sw) d\sw_t-\sigma_t(\omg) d\omega_t \right) \Big|   
\\ + \textstyle \frac{1}{2}\int_0^T \sum_{j=1}^d \Big[  \sum_{i=1}^m | \hat\lambda^i_t(\sw)|\  
|{\sigma}_t^{i,j}(\sw) -{\sigma}_t^{i,j}(\omg)| \Big]  \Big[  \sum_{i=1}^m |\hat\lambda^i_t(\sw)|\ 
|{\sigma}_t^{i,j}(\sw) +{\sigma}_t^{i,j}(\omg)| \Big]  dt     \Big]
 \end{multline*}
\begin{multline}\label{bla}
 \leq \textstyle K\EE^\pi\Big[\int_0^T \sum_{i=1}^m \left| b^i_t(\sw) - b^i_t(\omega)\right|   dt  +  \Big|  \int_0^T  \hat\lambda^*_t(\sw)   {\sigma}_t(\sw) ( d\sw_t-  d\omega_t ) \Big|  +  \Big|  \int_0^T  \hat\lambda^*_t(\sw)  ( {\sigma}_t(\sw)  - \sigma_t(\omg)) d\omega_t  \Big|  \\ \textstyle + 
 \frac{1}{2}\int_0^T \sum_{j=1}^d \left[  \sum_{i=1}^m 
|{\sigma}_t^{i,j}(\sw) -{\sigma}_t^{i,j}(\omg)| \right]  \left[  \sum_{i=1}^m  |{\sigma}_t^{i,j}(\sw) +{\sigma}_t^{i,j}(\omg)| \right]  dt    
 \Big].
\end{multline}
Now, by \eqref{eq:bLip} and since $\sw_0 -\omega_0 = 0$ on $\C^d$, we have that $\pi$-a.s.
\begin{equation*}\textstyle
\int_0^T\sum\limits_{i=1}^m \left| b^i_t(\sw) - b^i_t(\omega)\right|dt \leq L T m \sum\limits_{k=1}^d \,\,\sup\limits_{0 \leq t \leq T}  \big|  \sw^k_t -   \omega^k_t \big| \leq LTm \sum\limits_{k=1}^d V_T(\sw^k -  \omega^k) = LTm V_T(\sw -  \omega).
\end{equation*}
The second term in \eqref{bla} is easily bounded $\pi$-a.s.\ as follows
\begin{equation*}\textstyle
 \Big|  \int_0^T  \hat\lambda^*_t(\sw)  {\sigma}_t(\sw) (d\sw_t-  d\omega_t ) \Big| =  \sum_{j=1}^d\sum_{i=1}^m  \Big| \int_0^T  \hat{\lambda}^i_t(\sw)\sigma^{i,j}_t(\sw)(d\sw^j_t-  d\omega^j_t )  \Big| \leq  C m V_T(\sw -  \omega).
\end{equation*}
Finally we consider the third term in \eqref{bla}, and denote $ X_t = \int_0^t  \hat\lambda^*_s(\sw)  \left( {\sigma}_s(\sw)  - \sigma_s(\omg) \right) d\omega_s  $. This is a martingale under $\pi$, because by causality the coordinate process $\omega$ is an $\F \otimes \G$-martingale. Hence, we can apply the Burkholder-Davis-Gundy inequality, obtaining 

\begin{align*} 
\textstyle  \EXP^\pi[|X_T|] 
 & \leq \textstyle   C_1 \EXP^\pi\Big[\sqrt{\langle X,X \rangle_T}\Big]\\ 
 &\textstyle = \textstyle C_1 \EXP^\pi\left[\sqrt{\int_0^T \sum_{j=1}^d\Big | \sum_{i=1}^m \hat{\lambda}^i_t(\sw)(\sigma^{i,j}_t(\sw)-\sigma^{i,j}_t(\omega))\Big |^2 dt} \right]\\
 & \textstyle \leq  \textstyle  C_1 \sqrt{m}\, \EXP^\pi\left[ \sqrt{ \int\limits_0^T \sum\limits_{j=1}^d \sum\limits_{i=1}^m | \sigma^{i,j}_t(\sw)-\sigma^{i,j}_t(\omega)|^2 dt} \right]\\
 &\leq \textstyle C_1 M \sqrt{m d} \, \EXP^\pi\left[ \sqrt{ \int\limits_0^T \Big \{ \sum_{k=1}^d \sup_{0\leq s\leq t} | \sw^k_s-\omega^k_s |\Big \}^2 dt}
 \right]\\
& \leq \textstyle  C_1 M \sqrt{T m d} \, \EXP^\pi\left[ \sum_{k=1}^d \sup_{0\leq s\leq T} | \sw^k_s-\omega^k_s |\right] \\
& \textstyle \leq C_1 M \sqrt{T m d}\, \EE^\pi\left[V_T(\sw -  \omega) \right].
\end{align*}
Hence, the difference { $v^\G-v^\F$} is bounded above by
$\tilde{K}  \EE^\pi[V_T(\sw - \omega)]$, with 
\begin{equation}\label{eq:cnst}
\tilde{K}  = K(LTm +  m^2 d C M T +C m +  C_1 M \sqrt{T dm}).
\end{equation}
Since $\pi$ was a generic transport in $\textstyle\Pi^{\F,\G}(\gamma,\gamma)$, this concludes the proof for the case ${v^\G} < \infty$.
The case ${ v^\G} =\infty$ follows similarly, working along a sequence $\lambda^n$ s.t.\ $\EE[U(X_T^{\lambda^n})]\to \infty$.
\end{proof}

{
\begin{remark}
In the proof of Proposition~\ref{prop: ut}, the only place where we use the fact that one filtration is bigger than the other is to state that the difference of the two utility maximization problems is non-negative. Thus all estimates in \eqref{bound utility}, except the leftmost one, can be obtained in the same way for agents with any sets of informations (non-necessarily one bigger than the other).
\end{remark}
}

\subsection{Optional projections in causal transport}
\label{sec project causal}
In Sections~\ref{sec semimart Causal}, \ref{sec opt stop} and \ref{sect: utmax}, we intensively used projections and dual projections, which we briefly recall in~Appendix~\ref{app proj}. 
Here $\F$ and $\G$ are any two filtrations on $\C$, and for a process $M$ we denote by $^oM^{\F}$ (resp.\ $^{\F}M^o$) { its} optional projection (resp.\ dual optional projection) of $M$ with respect to $(\pi,\{\F\otimes \{\emptyset,\C\}\}^{\pi})$.
The essential difference between $^o\Lambda^{\F}$ and $^{\F}\Lambda^o$, as explained in \cite[Remark~VI.74-(c)]{DellacherieMeyerB}, is that while the first one formalizes $\EE[\Lambda_t|\F_t]$, the second one does so to $\int_0^t\EE[d\Lambda_s|\F_s]$. Causality { imposes a strong} relation between the two kinds of projections:
\begin{eqnarray*}\textstyle
	^{\F}\Lambda^o_t -\, ^{\F}\Lambda^o_0 &=& \textstyle\int_0^t\EE^{\pi}[d\Lambda_s|\F_s] = \int_0^t\EE^{\pi}[d\Lambda_s|\F_t] = \EE^\pi[\int_0^td\Lambda_s |\F_t]\\ &=&\textstyle \EE^{\pi}[\Lambda_t -\Lambda_0|\F_t] = \,
	^o\Lambda^{\F}_t -\, ^o\Lambda^{\F}_0.
\end{eqnarray*}
This results is formalised in the following proposition, and was crucial in our applications to optimal stopping problems in Section~\ref{sec opt stop}. 
Such a phenomenon is not symmetric, just as causality, i.e., one does not expect Proposition~\ref{prop crucial} to hold for projections w.r.t. $\G$. 
\begin{proposition}
\label{prop crucial}
Let $\pi\in \Pi^{\F ,\G }(\mu, \nu)$ be a causal transport plan, and let $\Lambda$ be an $(\F\otimes\G)$-adapted, integrable variation \cadlag\ process with $\Lambda_0=0$. Then $$^{\F}\Lambda^o\,\,\mbox{ is $\pi$-indistinguishable from } \,\, ^o\Lambda^{\F}.$$
\end{proposition}

\begin{proof}
We drop the superscript $\F$ to simplify the notation.
Fix $t\in[0,T]$, and consider { the process $X=(X_s)_{s\in[0,t]}$ with constant paths given by}
\[
X_s=\IND_{\{^o\Lambda_t>\Lambda_t^o\}},\qquad s\in[0,t].
\]
Its optional projection with respect to $(\pi,\{\F\otimes \{\emptyset,\C\}\}^{\pi})$ satisfies
\begin{equation*}\label{eq:mgc}
^oX_s=\pi[^o\Lambda_t>\Lambda_t^o|\{\F_s\otimes \{\emptyset,\C\}\}^{\pi}],\qquad s\in[0,t].
\end{equation*}
Note that $(^oX_s)_{s\in[0,t]}$ is an $\{\F\otimes \{\emptyset,\C\}\}^{\pi}$-martingale, hence $\pi$-indistinguishable from an $(\F\otimes \{\emptyset,\C\})$-martingale, by \cite[Lemma~7, Appendix~I]{DellacherieMeyerB}, which is then also an $(\F\otimes \G)$-martingale by causality; see Remark~\ref{rem: causal}. { This means that $(^oX_s)_{s\in[0,t]}$ 
is the \cadlag\ version of the $(\F\otimes \G)$-martingale $M_s=\pi[^o\Lambda_t>\Lambda_t^o|\F_s\otimes\G_s]$}, $0\leq s\leq t$, thus $\pi$-indistinguishable from the optional projection of $X$ w.r.t.\ $(\pi,\{\F\otimes\G\}^\pi)${, which we denote by $Y$}.

Now, by definition of optional projection,
\[
\EE^\pi[\Lambda_tX_t]=\EE^\pi[^o\Lambda_tX_t].
\]
On the other hand, since $X$ is constant and $\Lambda_0=0$, 
\[
\EE^\pi[\Lambda_tX_t] = \EE^\pi[\Lambda_tX_0]= \EE^\pi\left[\int_0^tX_sd\Lambda_s\right].
\]
Now, using \cite[Remark~VI.58-(d)]{DellacherieMeyerB} and the fact that $\Lambda$ is an $(\F\otimes\G)$-adapted integrable variation process, we have that
$\EE^\pi\left[\int_0^tX_sd\Lambda_s\right]=
\EE^\pi\left[\int_0^tY_sd\Lambda_s\right]$.
Therefore we have
\begin{eqnarray*}\textstyle
\EE^\pi[\Lambda_tX_t] &=&\textstyle \EE^\pi\left[\int_0^tY_sd\Lambda_s\right]=\textstyle \EE^\pi\left[\int_0^t{}^oX_sd\Lambda_s\right]
=\EE^\pi\left[\int_0^tX_sd\Lambda^o_s\right]\\
&=&\EE^\pi\left[\Lambda^o_t \,X_t\right]-
\EE^\pi\left[\Lambda^o_0\, X_t\right]=\textstyle\EE^\pi\left[\Lambda^o_t\, X_t\right],
\end{eqnarray*}
where in the second equality we use that $^oX$ and $Y$ are $\pi$-indistinguishable, and in the third one the definition of dual optional projection (recall $\Lambda^o_0=0$, see Appendix~\ref{app proj}). Altogether we have 
\[ \textstyle
\EE^\pi[^o\Lambda_t\IND_{\{^o\Lambda_t>
\Lambda_t^o\}}]=
\EE^\pi[\Lambda_t^o\IND_{\{^o\Lambda_t>
\Lambda_t^o\}}],
\]
hence $\{^o\Lambda_t>\Lambda_t^o\}$ is $\pi$-negligible. Arguing similarly, we get  that $\{^o\Lambda_t\neq\Lambda_t^o\}$ is $\pi$-negligible. Since this is true for all $t\in[0,T]$, we have that  $^o\Lambda$ and $\Lambda^o$ are versions of each other and hence $\pi$-indistinguishable since both are \cadlag\ (see \cite[Theorem~VI.47]{DellacherieMeyerB}).
\end{proof}

After presenting this work, we learned about the preprint~\cite{AL16}, where the following is proved: if two filtrations $\A^1\subseteq\A^2$ satisfy the $H$-hypothesis, then for any $\A^2$-optional process of integrable variation, its $\A^1$-optional and $\A^1$-dual optional projections coincide. Thanks to Remark~\ref{rem: causal}, Proposition~\ref{prop crucial} follows by \cite[Theorem~2]{AL16}.

\section{Attainability and Duality}\label{sect:dual}

\subsection{Classical and constrained transport: an extension}\label{sect:dualcl}
In this section we consider the general abstract setting of Section \ref{sec sub abstract}.
 As we have seen in previous sections, it is important to obtain attainability and duality results for \eqref{eq classical transport}, and more specifically \eqref{eq causal transport}, when the cost function $c$ is Borel measurable with values in the extended real line $(-\infty,+\infty]$; e.g.\ for Cameron-Martin or total-variation costs. For such problems there is no systematic theory, and indeed \cite[Example 4.1]{MathiasWalter} shows that duality may fail in such a setting. Fortunately, the cost functions we shall encounter in this article have a strong structural property. Assuming this property will allow us to prove attainability/duality results for \eqref{eq classical transport}-\eqref{eq causal transport} in a simple and self-contained way. The next result can also be proven via the following argument: on every Polish space, there is a finer Polish  topology having the same Borel sets, for which a given real-valued Borel function becomes continuous. This argument, however, would not lead us to prove Corollary \ref{coro constraints} below, nor help us studying our ultimate object, namely \eqref{eq causal transport}. So we rather give our own arguments below.

\begin{proposition}\label{prop extension Kellerer-Mathias}
Let $f:\X\to\R$ and $g:\Y\to\R$ be bounded Borel functions and $\tilde{c}:\X\times\Y\to(-\infty,\infty]$ lower semicontinuous and bounded from below. Suppose that either $f$ or $g$ is further continuous, and define $$c(x,y):=\tilde{c}(x,y)+ f(x)g(y).$$ Then the optimal transport problem \eqref{eq classical transport} corresponding to the cost $c$ is attained. Furthermore, there is no duality gap:
\begin{equation*}\label{eq classical transport product}
\textstyle
\inf\limits_{\pi\in\Pi(\mu,\nu)}\EXP^{\pi} [\, c \, ] = \sup\limits_{\substack{\phi\in\C_b(\X),\,\psi\in\C_b(\Y) \\ \phi\oplus\psi\leq \tilde{c}+ fg }} \left\{\int \phi(x)\mu(dx)+\int \psi(y)\nu(dy) \right \}.
\end{equation*}
\end{proposition}

Before {providing the proof of the above proposition, we introduce} the following lemma.

	\begin{lemma}\label{lem:fhcts}
		Assume that $\X, \Y$ are Polish spaces equipped with Borel probability measures $\mu, \nu$. Let $f:\X\to\R$ and $g:\Y\to\R$ be bounded Borel function, at least one of which is continuous. Then the function
		\[
		\pi \mapsto \EE^\pi[f(x)g(y)]
		\]
		is continuous on $\Pi(\mu,\nu)$ with respect to the weak topology. 
	\end{lemma}
	
	\begin{proof}
		Recall that $\Pi(\mu,\nu)$ is a compact subset of $\P(\X\times\Y)$ with respect to the weak topology. W.l.o.g.\ we assume that $f$ is continuous and let $M\in\R$ be such that $|f|\leq M$, and $g_k:\Y\to\R, k\in\NN$, be a sequence of bounded continuous functions converging to $g$ in $L^1(\nu)$.
		Consider a sequence of measures $(\pi_n)_n\subseteq \Pi(\mu,\nu)$ such that $\pi_n$ converges weakly to $\pi$ for some $\pi\in\Pi(\mu,\nu)$.
		For any $\epsilon>0$, take $k(\epsilon)$ such that $\|g-g_k\|_{L^1(\nu)}\leq\epsilon/(2M+1)$ for all $k\geq k(\epsilon)$, and take $n(\epsilon)$ such that $|\EE^\pi[fg_{k(\epsilon)}]-\EE^{\pi_n}[fg_{k(\epsilon)}]|\leq\epsilon/(2M+1)$ for all $n\geq n(\epsilon)$. Then, for all $n\geq n(\epsilon)$,
		\begin{eqnarray*}
			\left|\EE^\pi[fg]-\EE^{\pi_n}[fg]\right|&\leq& \EE^\pi[|f||g-g_{k(\epsilon)}|]+\left|\EE^\pi[fg_{k(\epsilon)}]-\EE^{\pi_n}[fg_{k(\epsilon)}]\right|+\EE^{\pi_n}[|f||g-g_{k(\epsilon)}|]\\
			&\leq& M\epsilon/(2M+1)+\epsilon/(2M+1)+
			M\epsilon/(2M+1)\leq\epsilon,
		\end{eqnarray*}
		which proves the desired statement.
	\end{proof}

	\begin{proof}[Proof of Proposition \ref{prop extension Kellerer-Mathias}]
		We employ classical arguments, as in \cite{Villani} or \cite[Sect.~1.3]{MathiasWalter}.
		Since $\tilde c$ is lower semicontinuous, there exists
		a sequence $(\tilde{c}_n)_n$ of bounded continuous functions on $\X\times\Y$ such that $\tilde{c}_n\uparrow \tilde{c}$.  We are going to show that $P(c_n):=\inf_{\pi\in\Pi(\mu,\nu)}\EE^\pi[c_n]$ converges to $P(c):=\inf_{\pi\in\Pi(\mu,\nu)}\EE^\pi[c]$, where $c_n(x,y)=\tilde{c}_n(x,y)+f(x)g(y)$.
		In order to do so, for each $n\in\NN$, we pick $\pi_n\in\Pi(\mu,\nu)$ such that $\EE^{\pi_n}(c_n)\leq P(c_n)+1/n$. Since $\Pi(\mu,\nu)$ is weakly compact, we may assume that $(\pi_n)_n$ converges weakly to some transport plan $\pi\in\Pi(\mu,\nu)$. Then,
		\begin{eqnarray*}
			P(c)&\leq& \EE^\pi[c]=\lim_m\EE^\pi[c_m]=\lim_m(\lim_n\EE^{\pi_n}[c_m])\leq\lim_m(\lim_n\EE^{\pi_n}[c_n])\\
			&=&\lim_n\EE^{\pi_n}[c_n]=\lim_n P(c_n)\leq P(c),
		\end{eqnarray*}
		where we used monotone convergence, Lemma~\ref{lem:fhcts} to ensure that $\EE^\pi[c_m]=\lim_n\EE^{\pi_n}[c_m]$, and the facts that $c_n$ is an increasing sequence with $P(c_n)\leq\EE^{\pi_n}(c_n)\leq P(c_n)+1/n$. This concludes the proof of our claim and actually shows that $\pi$ is an optimizer for the cost $c$ (this easily follows by compactness and Lemma~\ref{lem:fhcts}). 
		The function $c_n$ is Borel bounded, so by \cite[Theorem 2.14]{KellererDuality} we have that duality holds for it. Thus we can pick $(\psi_n,\phi_n)$ such that $\int \psi_nd\mu+\int\phi_nd\nu\geq P(c_n)-1/n$, hence $\sup_n(\int\psi_nd\mu+\int\phi_nd\nu)\geq\lim_nP(c_n)-1/n=P(c)$.
		Since $\psi_n(x)+\phi_n(y)\leq c_n(x,y)\leq c(x,y)$, duality is established.
	\end{proof}
	
{ It is clear} that in Proposition~\ref{prop extension Kellerer-Mathias} one can take $c$ to contain a finite sum of terms of the form $f(x)g(y)$ as described. We give now a corollary of this proposition, dealing with a class of optimal transport problems under linear constraints. It is this result that we shall later apply to the setting of causal optimal transport. We refer to \cite{Zaev,BGriessler} for more on linearly constrained transport problems, but remark that the result below is not a consequence of theirs. 
\begin{corollary}\label{coro constraints}
Let $\mathfrak{F}$ (resp.\ $\mathfrak{G}$) be a non-empty collection of real-valued bounded Borel functions on $\X$ (resp.\ $\Y$), and define\footnote{Here and thereafter, \textit{span} denotes the linear space of functions obtained by finite linear combinations of those functions in the generating class.} $\mathfrak{H}^{\mathfrak{F},\mathfrak{G}}:=\text{span}\{fg:\,\, f\in \mathfrak{F},  g\in \mathfrak{G}\}$.
We define the optimal transport with linear constraints determined by $ \mathfrak{H}^{\mathfrak{F},\mathfrak{G}}$ as
\begin{equation}\label{eq transport constraints}\textstyle
\inf\limits_{\substack{\pi\in\Pi(\mu,\nu)\\ \forall h\in \mathfrak{H}^{\mathfrak{F},\mathfrak{G}}: \,\,\EXP^{\pi}[h]=0}}\EXP^{\pi} [\, c \, ].
\end{equation}
Assume that ${c}:\X\times\Y\to(-\infty,\infty]$ is lower semicontinuous and bounded from below, and that either all elements of $\mathfrak{F}$, or all elements of $\mathfrak{G}$, are continuous. Then \eqref{eq transport constraints} is attained, and there is no duality gap:
\begin{equation*}\label{eq transport constraints duality}\textstyle
\inf\limits_{\substack{\pi\in\Pi(\mu,\nu)\\ \forall h\in \mathfrak{H}^{\mathfrak{F},\mathfrak{G}}: \,\,\EXP^{\pi}[h]=0}}\EXP^{\pi} [\, c \, ]= \sup\limits_{\substack{\phi\in\C_b(\X),\,\psi\in\C_b(\Y),\, h\in   \mathfrak{H}^{\mathfrak{F},\mathfrak{G}} \\ \phi\oplus\psi\leq {c}+h }} \left\{\int \phi(x)\mu(dx)+\int \psi(y)\nu(dy) \right \}.
\end{equation*}
\end{corollary}

\begin{proof}
The set of all $\pi\in\Pi(\mu,\nu)$ s.t.\  for all $ h\in \mathfrak{H}^{\mathfrak{F},\mathfrak{G}}$ it holds that $\EXP^{\pi}[h]=0$, is a weakly closed subset of the compact set $\Pi(\mu,\nu) $ (follows from Lemma~\ref{lem:fhcts}). This is enough to ensure the attainability of  \eqref{eq transport constraints}.
Further, it is immediate that 
$$\inf_{\substack{\pi\in\Pi(\mu,\nu)\\ \forall h\in \mathfrak{H}^{\mathfrak{F},\mathfrak{G}}: \,\,\EXP^{\pi}[h]=0}}\EXP^{\pi} [\, c \, ] = \inf_{\pi\in\Pi(\mu,\nu)}\sup_{ h\in \mathfrak{H}^{\mathfrak{F},\mathfrak{G}}}\EXP^{\pi} [\, c+h \, ] =  \sup_{ h\in \mathfrak{H}^{\mathfrak{F},\mathfrak{G}}}\inf_{\pi\in\Pi(\mu,\nu)}\EXP^{\pi} [\, c+h \, ], $$
by the usual minimax arguments (e.g.\ \cite{Sion}, and observe that the affine bilinear objective functions is lower semicontinuous in $\pi$ by Lemma~\ref{lem:fhcts}, as $\pi$ varies over a compact). By Proposition~\ref{prop extension Kellerer-Mathias} we find
$$\inf_{\substack{\pi\in\Pi(\mu,\nu)\\ \forall h\in \mathfrak{H}^{\mathfrak{F},\mathfrak{G}}: \,\,\EXP^{\pi}[h]=0}}\EXP^{\pi} [\, c \, ] = \sup_{ h\in \mathfrak{H}^{\mathfrak{F},\mathfrak{G}}} \,\,\, \sup_{\substack{  \phi\in \C_b(\X),\psi\in \C_b(\Y)\\ \phi  \oplus \psi\leq c+h } }\left\{\int \phi(x)\mu(dx)+\int \psi(y)\nu(dy) \right \},$$
yielding the desired result.
\end{proof}

\subsection{Attainability and Duality in causal transport}\label{sect:dualc}
We recall from \eqref{eq causal transport} that by a causal optimal transport problem with respect to a cost function $c:\X\times\Y\to (-\infty,\infty]$, we mean the optimization problem
\begin{equation*}\textstyle
\inf\limits_{\pi\in\Pi^{\F^\X,\F^\Y}(\mu,\nu)}\EXP^{\pi} [\, c \, ].
\end{equation*}
It has already been observed that these problems form a subclass of optimal transport problems under linear constraints; see \cite{Lassalle2,BBLZ}. Let us make this precise in the present setting, by defining 
\begin{equation*}
\textstyle \fH\,\,:=\,\, span\left (\left\{ g\left[f- \EXP^\mu[ f| \F^\X_t]\right]\, : \,f\in C_b(\X),g\in B_b(\Y,\F^\Y_t),t\in [0,T]\right \} \right ).\label{eq h}
\end{equation*} 

\begin{lemma}\label{lem characterization causality}
Let $\pi\in \Pi(\mu, \nu)$. Then $\pi$ is causal w.r.t. $\F^\X$ and $\F^\Y$ (i.e.\ $\pi\in \Pi^{\F^\X,\F^\Y}(\mu, \nu)$) if and only if $ \EXP^\pi[\,h\,]=0$ for all $h\in\fH$.
\end{lemma}

\begin{proof}
For $g$ bounded $(\Y,\F_t^\Y)$-measurable, denote $g^t= g^t(x): = \EXP^\pi[ g(y) |\F_T^\X\otimes \{\emptyset,\C\} ](x)$. By definition $\pi$ is  causal w.r.t. $\F^\X$ and $\F^\Y$ if and only if for all $t\leq T$ and all such $g^t$ we have
\begin{align}\notag
g^t = \EXP^\mu[ g^t |\F^\X_t],\,\,\,\, \mu-a.s, 
\end{align}
which is equivalent to
\begin{align*}
\EXP^\mu [ f ( g^t -  \EXP^\mu[ g^t |\F^\X_t] ) ]= 0,
\end{align*}
for every continuous bounded function $f :\X\to\R$ and for all $t\leq T$. The fact that we can take the $f$'s continuous and not merely measurable comes from the fact that $\mu$ is a Borel finite measure on a Polish space. It is easy to see that the previous equation is equivalent to 
\begin{align*}
\EXP^\mu[ g^t ( f -  \EXP^\mu [f|\F_t^\X] ) ]= 0.
\end{align*}
Finally, by the tower property of conditional expectations the latter is in turn equivalent to 
\begin{align*}
\EXP^\pi [ g ( f-  \EXP^\mu [f|\F_t^\X] ) ] =0. 
\end{align*}
\end{proof}

In \cite[Sect. 3]{Lassalle2} the author proves, in the general setting of Polish spaces, that the set $\Pi^{\F^\X,\F^\Y}(\mu, \nu)$ is closed for weak convergence. Thus an attainability and duality theory for the problem of optimal transport under the causality constraint follows. This is done there at the expense of a regularity assumptions of sorts on the filtration $\F^\Y$ (see \cite[Definition 3]{Lassalle2}). Such an assumption would, in our context, drastically limit the applicability of the transport approach. As we now show, we can still obtain an attainability and duality theory without assumptions on $\F^\Y$. We do this at the price of requiring the first marginal $\mu$ to be ``weakly continuous'' in a precise sense. This is enough for the purpose of our work.

\begin{lemma}\label{lem:cpt}
Assume that $\mu$ satisfies the following weak continuity property:
\begin{equation}\label{eq:weak_cont}
\forall f\in C_b(\X),t\in[0,T]:\,\, x\mapsto \EXP^\mu[f|\F^\X_t](x) \mbox{ is continuous}.
\end{equation}
Then the set $\Pi^{\F^\X,\F^\Y}(\mu,\nu)$ of causal couplings is compact for weak convergence.
\end{lemma}
\begin{proof}
Since $\Pi^{\F^\X,\F^\Y}(\mu,\nu)\subseteq\Pi(\mu,\nu)$, and the latter is weakly compact, we only need to show that $\Pi^{\F^\X,\F^\Y}(\mu,\nu)$ is weakly closed. This follows from Lemma~\ref{lem characterization causality}, since the condition $\EXP^\pi[\,h\,]=0$ is closed for all $h\in\fH$, by Lemma~\ref{lem:fhcts} and the continuity property \eqref{eq:weak_cont}. 
\end{proof}

\begin{theorem}[Causal transport duality]\label{prop duality}
Let $c:\X\times\Y\to (-\infty,\infty]$ be bounded from below and lower semicontinuous. Further assume that $\mu$ satisfies the weak continuity property~\eqref{eq:weak_cont}.
Then \eqref{eq causal transport} is attained and there is no duality gap:
\begin{align*} \textstyle
\inf\limits_{\pi\in\Pi^{\F^\X,\F^\Y}(\mu,\nu)}\EXP^{\pi} [\, c \, ]
=\sup\limits_{\substack{  \phi\in C_b(\X),\,\psi\in C_b(\Y),\, h\in \fH \\ \phi\oplus\psi\leq c+h  }} \left\{\int\phi d\mu+\int \psi d\nu \right \} 
= \sup\limits_{\substack{  \psi\in C_b(\Y),\, h\in \fH \\ \psi\leq c+h}} \left\{\int \psi d\nu \right \} .
\end{align*}
\end{theorem}

\begin{proof}
Attainability follows by classical arguments from Lemma~\ref{lem:cpt}.
Duality is a direct consequence of Lemma~\ref{lem characterization causality} and Corollary~\ref{coro constraints}, upon observing that $\fH$ has the correct structure and that, under the weak continuity assumption on $\mu$, all the $x$-dependent factors generating $\fH$ are continuous and bounded. The fact that $\phi$ disappears from the dual problem follows from the fact that $\phi-\EXP^\mu[\phi]$ belongs to $\fH$.
\end{proof}

The previous weak continuity property of $\mu$ is fulfilled if e.g.\ $\X$ is a path space and $\mu$ is the law of a Feller process. So the case that interests us, Wiener measure on continuous path space, is fully covered.

\appendix

\section{Elements of Orlicz space theory}\label{Sec orlicz}

As presented in \cite{RaoRen}, a convex even function $\Phi:\R\to \R_+\cup\{+\infty\}$ satisfying $\Phi(0)=0$ and $ \Phi(\infty)=\infty$, is called a \textit{Young function}. If such a function is finite-valued, it is zero only at the origin, and satisfies both $\Phi(0)/0=0$ and $\Phi(\infty)/\infty=\infty$, then it is called an \textit{N-function}. We remark that $\Phi$ is a Young function (resp.\ N-function) if and only if its conjugate $\Phi^*$ is so.  

From now on we identify processes which are $d\nu\times dt$-a.e.\ equal. Assuming that $\Phi$ is a Young function, we define
$$\textstyle M^{\Phi} \, :=\,\left \{F:\C\times [0,T]\to \R, \mbox{ s.t. $F$ is $\G$-previsible and }\forall k>0:\, \int\int_0^T \Phi(kF_t(\sw))dt\nu(d\sw)<\infty\right  \},$$
which is a closed subspace (sometimes called Orlicz heart or Morse-Transue space) of the so-called Orlicz space
 $$\textstyle L^{\Phi} \, :=\,\left \{F:\C\times [0,T]\to \R, \mbox{ s.t. $F$ is $\G$-previsible and }\exists k>0:\, \int\int_0^T \Phi(kF_t(\sw))dt\nu(d\sw)<\infty\right  \},$$
when endowed with the \textit{gauge norm} 
  $$\textstyle \|F\|_{\Phi}:=\inf \left\{\beta >0: \int \int_0^T \Phi(F_t(\sw)/\beta)dt\,\nu(d\sw)\leq 1\right\}.$$
The gauge norm actually {turns $L^{\Phi}$} into a Banach space, and if e.g.\ $\Phi$ is an N-function then the norm-dual of $M^{\Phi}$ is $L^{\Phi^*}$ (\cite[Ch. III.3.3, Theorem 10]{RaoRen} and \cite[Ch. IV.4.1, Theorem 6]{RaoRen}). 

We now introduce growth conditions on $\Phi$ and $\Phi^*$. We say that a Young function $\Phi$ is in $\Delta_2 $ if there are some $C,x_0>0$ s.t.\ whenever $x\geq x_0$ we have $\Phi(2x)\leq C\Phi(x)$. This is seen equivalent to the following condition on $\Phi^*$: there exist $\ell>1,x_0>0$ s.t.\ whenever $x\geq x_0$ we have $\Phi^*(x)\leq \frac{\Phi^*(\ell x)}{2\ell}$. When $\Phi$ is an N-function, then by \cite[Ch. II.2.3, Theorem 3]{RaoRen}:
\begin{equation}\label{eq quantitative Delta 2}
\textstyle \Phi\mbox{ in } \Delta_2 \iff \exists x_0>0,\epsilon>1 \mbox{ s.t. } \left[x\geq x_0 \Rightarrow \frac{x\Phi'(x)}{\Phi(x)}\leq \epsilon \right].
\end{equation}
Clearly when $\Phi$ is in $\Delta_2$ we have $M^\Phi=L^\Phi$. The reflexivity of $L^\Phi$ is essentially equivalent to $\Phi$ and $\Phi^*$ being in $\Delta_2$.

We finally provide a technical lemma useful in the proof of Theorem~\ref{thm orlicz}:

\begin{lemma}\label{lem aux Sobolev}
If $\rho^*$ is an N-function in $\Delta_2$, then
$$ \textstyle (\omega,\sw)\in \C\times \C \mapsto \int_0^T\rho(\abs{\sw_t-\omega_t})dt$$
is lower semicontinuous when $\C\times \C$ is equipped with the ``sum of the uniform norms'' norm.
\end{lemma}

\begin{proof}
It suffices to show the lower semicontinuity of $\sw\mapsto r(\sw):= \int_0^T\rho(\abs{\sw_t})dt$. Let $\C_K:=\{\sw\in \C:r(\sw)\leq K\}$. Since $\rho^*$ is in $\Delta_2$, we have by \cite[Ch. II.2.3, Corollary 5]{RaoRen} that $\rho$ grows at least as fast as some power function with exponent $p> 1$. This shows that $\C_K$ is bounded in $W^{1,p}([0,T])$, the Sobolev space of absolutely continuous functions with $p$-integrable first derivative. By classical arguments we have that if $\sw^n\to\sw$ uniformly, with $\sw^n\in \C_K$, then $\sw\in W^{1,p}([0,T])$ and in particular $\sw$ is absolutely continuous too. By Fatou's Lemma we further get $\sw\in \C_K$, yielding the desired result.
\end{proof}

\section{Proof of Theorem \ref{thm orlicz}}\label{sec Orlicz proof}
We consider a fixed measure $\nu$ on $\C$.
The idea behind the refined dual \eqref{eq: dual easy} comes from the next argument:

\begin{lemma}\label{lem FY}
Let $\rho:\R\to\R\cup\{+\infty\}$ be convex.
Weak-refined duality holds, in the sense that the value of the primal problem \eqref{eq: finiteness} can only be larger than that of the refined dual \eqref{eq: dual easy}.
\end{lemma}

\begin{proof}
W.l.o.g.\ we assume that \eqref{eq: finiteness} is finite.  Let $\psi(\sw) = \int_0^T {F}_t(\sw)d\sw_t - \int_0^T \rho^*({F}_t(\sw))dt$ and $h(\omega,\sw)= \int_0^T {F}_t(\sw)d\omega_t$, with $F\in S_a(\G)$. Taking any $\pi\in\Pi^{\F,\G}(\gamma,\nu)$ with finite cost, and denoting $\abs{\sw_t-\omega_t}=\alpha_t$, we have
\begin{equation*}\textstyle
\psi(\sw) \leq   \int_0^T\rho(\alpha_t)dt + h(\omega,\sw)  =\int_0^T\rho(\abs{\sw_t-\omega_t})dt + h(\omega,\sw) \,\,\,\, \pi\text{-a.s},
\label{eq FY}
\end{equation*} 
by Fenchel-Young inequality. Since by causality $\EE^\pi[h]=0$, we conclude that the value of the primal problem is not any less than \eqref{eq: dual easy}. 
\end{proof}
 We now define a stochastic integral; see Appendix~\ref{Sec orlicz} for terminology and notation on Orlicz spaces, such as ``N- and Young-functions'', the ``$\Delta_2$ condition'' and so forth.

\begin{lemma}\label{lem orlicz}
Suppose that $\rho^*$ is a Young function having a global minimum at the origin. Then 
\begin{align}\textstyle
\overline{S_a(\G)}^{\|\cdot\|_{\rho^*}}&=M^{\rho^*} \label{eq orlicz}.
\end{align}
If further the refined dual \eqref{eq: dual easy} is finite, then the functional $$\textstyle F=\sum F^i\IND_{(\tau_i,\tau_{i+1}]}\in S_a(\G) \mapsto \int\int_0^T F_td\sw_t\,\nu(d\sw):= \int\sum F^i(\sw)[\sw_{\tau_{i+1}}- \sw_{\tau_{i}}] \nu(d\sw)$$
can be uniquely extended to $M^{\rho^*}$ by continuity. Using the same notation for its extension, we can replace the optimization variables in \eqref{eq: dual easy} by taking ``$F\in M^{\rho^*}$'' without changing the value of the optimization problem.
\end{lemma}

\begin{proof}
By \cite[Ch. III.3.4, Proposition 3]{RaoRen}, we have \eqref{eq orlicz} under the given hypotheses; the measure considered being $d\nu\times dt$, the sigma-algebra being the $\G$-previsible one, and the Young function being $\rho^*$. One need only observe that the previsible sigma-algebra is generated by the algebra of sets, whose elements are finite disjoint unions of ``base'' sets of the form $D\times (\underline{\tau},\overline{\tau}]$, with $D\in \G_{\underline{\tau}}$ and $\G$-stopping times $\underline{\tau}\leq\overline{\tau}$. If we denote by $v$ the value of problem \eqref{eq: dual easy}, from now on assumed finite, we then have for all $F\in S_a(\G),\beta >0$:
\begin{equation*}\textstyle
\int\int_0^T (F_t/\beta)d\sw_t\,\nu(d\sw)\leq v + \int \int_0^T \rho^*(F_t(\sw)/\beta)dt\,\nu(d\sw),  \label{eq cont bound}
\end{equation*} 
so by definition 
\begin{equation}\label{eq cont bound 2}\textstyle
\int\int_0^T F_t\, d\sw_t\,\nu(d\sw)\leq (v+1)\|F\|_{\rho^*}.
\end{equation}
This shows that the discrete integral, seen as a continuous linear functional on $S_a(\G)$, can be uniquely and continuously extended to the norm closure of this space, which we know to coincide with $M^{\rho^*}$. Because the convex functional { $F\mapsto  \int\int_0^T\rho^*( F_t)\, dt \,\nu(d\sw)$} is finite throughout $M^{\rho^*}$, for all $F, \bar{F}\in  M^{\rho^*}$ with $\|F-\bar{F}\|_{\rho^*}\leq 1/2$ we have that
$$\textstyle \int\int_0^T\rho^*( \bar{F}_t)\, dt \,\nu(d\sw) \leq 1/2 \int\int_0^T\rho^*( 2 F_t)\, dt \,\nu(d\sw) + 1/2\int\int_0^T\rho^*\left(\frac{ F_t-\bar{F}_t}{1/2}\right)\, dt \,\nu(d\sw)\leq c_F+1,$$
where $c_F$ is a constant only depending on $F$.
This shows that the convex functional is locally bounded and thus continuous by classical results; the last statement then follows.
\end{proof}

\begin{lemma}\label{lem attain}
Assume that $\rho^*$ (equiv.\ $\rho$) is an N-function, that $\rho$ is in $\Delta_2$, and that the refined dual \eqref{eq: dual easy} is finite. Then \eqref{eq: dual easy} is attained in $M^{\rho^*}$.
\end{lemma}

\begin{proof}
By Lemma \ref{lem orlicz}, we can consider \eqref{eq: dual easy} as defined over $M^{\rho^*}$. By  \cite[Ch. IV.4.1, Theorem 6]{RaoRen} we see that $L^{\rho^*}$ is the norm dual space of $M^\rho$, so the classical Banach-Alaoglu's theorem implies that closed balls in $M^{\rho^*}$ are $\sigma(M^{\rho^*},M^{\rho})$-compact, since $M^{\rho^*}$ is a norm-closed and convex subset of $L^{\rho^*}$. By \cite[Ch. V.5.3, Theorem 3]{RaoRen} and the comment following its proof, we see that $\rho$ in $\Delta_2$ implies that 
\begin{equation}\label{eq lim Frho}\textstyle
\lim\limits_{\|F\|_ {\rho^*}\to \infty}\frac{\int\int_0^T\rho^*(F_t)dt\,\nu(d\sw)}{\|F\|_ {\rho^*}}=+\infty.
\end{equation}
As a consequence of this and \eqref{eq cont bound 2} (which holds also in $M^{\rho^*}$), we get
$$\textstyle \int [\int_0^T F_td\sw_t - \int_0^T \rho^*(F_t)dt ] \nu(d\sw) \leq \|F\|_ {\rho^*}\big\{ v+1 - \frac{\int\int_0^T\rho^*(F_t)dt\,\nu(d\sw)}{\|F\|_ {\rho^*}} \big\},$$
where $v$ is the value of \eqref{eq: dual easy}, and by \eqref{eq lim Frho} the r.h.s.\ above goes to $-\infty$ as $\|F\|_ {\rho^*}\to \infty$. This shows that in computing \eqref{eq: dual easy} one may restrict the problem to a big enough fixed ball in $M^{\rho^*}$, which is $\sigma(M^{\rho^*},M^{\rho})$-compact. As in the end of the proof of Lemma~\ref{lem orlicz}, we observe that the objective function is norm-continuous, and because it is concave it is also $\sigma(M^{\rho^*},M^{\rho})$-upper semicontinuous. The existence of an optimizer in $M^{\rho^*}$ follows.
\end{proof}

\begin{lemma}\label{lem many}
Suppose \eqref{eq: dual easy} is finite and attained by some $\hat{F}\in M^{\rho^*}$, and that $\rho^*$ is a differentiable N-function (equiv.\ $\rho$ is a strictly convex N-function) which is in $\Delta_2$. Setting $$\textstyle\alpha_t(\sw):= (\rho^*)'(\hat{F}_t(\sw))\,\mbox{ and }\,\,\xi_t(\sw):=\sw_t-\int_0^t \alpha_t(\sw)dt,$$ we have that $\int\int_0^T \rho(\alpha_t(\sw)) dt\,\nu(d\sw)< +\infty$, and so $\alpha\in L^\rho$ and $\xi$ is a $(\nu,\G)$-martingale. 

If further $\nu\ll\gamma$, then $\xi$ is a $(\nu,\G)$-Brownian motion, $(\xi,id)_\#\nu\in \Pi^{\F,\G}(\gamma,\nu)$, the primal problem \eqref{eq: finiteness} is finite and its value is equal to \eqref{eq dual wiener} and \eqref{eq: dual easy}. 
\end{lemma}

\begin{proof}
We first observe that by the identity of sub-differentials
\begin{align}\textstyle \label{eq integ alpha}\int\int_0^T \rho\circ (\rho^*)' (\hat{F}_t(\sw)) dt\,\nu(d\sw)= - \int\int_0^T  \rho^* (\hat{F}_t(\sw)) dt\,\nu(d\sw) + \int\int_0^T\hat{F}_t(\sw) (\rho^*)' (\hat{F}_t(\sw)) dt\,\nu(d\sw),
\end{align}
so the finiteness of the l.h.s.\ is equivalent to the finiteness of the second term in the r.h.s. Since $\rho^*$ is an N-function in $\Delta_2$ we have by \eqref{eq quantitative Delta 2} that when $\hat{F}_t(\sw)$ is large the integrand $\hat{F}_t(\sw) (\rho^*)' (\hat{F}_t(\sw))$ is dominated by a (fixed) constant times $\rho^*(\hat{F}_t(\sw))$, and so we conclude that the left- and right-hand sides above are indeed finite. In particular $\alpha\in L^\rho$ holds, and $h\in M^{\rho^*}\mapsto \int \int_0^T \alpha_t(\sw)h_t(\sw) dt \nu(d\sw)$ is finite-valued and continuous by \cite[Ch. III.3.3, Proposition 1]{RaoRen}.

	Since \eqref{eq: dual easy} is a concave problem, we have that if $\hat{\zeta}$ is optimizer of $\sup\limits_{\zeta \in M^{\rho^*}} H(\zeta)$, then  $\forall h\in M^{\rho^*}$ it holds that $\frac{\partial}{\partial \varepsilon} H(\hat{\zeta} + \varepsilon h) \vert_{\varepsilon = 0} = 0$, where
	$$\textstyle H(\zeta):= [\int\int_0^T \zeta_t(\sw)\,\, d\sw_t \nu(d\sw) - \int \int_0^T \rho^*(\zeta_t(\sw)) dt \nu(d\sw)].$$  
	Thus we get that $$\textstyle\int\int_0^T h_t(\sw) d\sw_t \nu(d\sw) - \int \int_0^T \alpha_t(\sw)h_t(\sw) dt \nu(d\sw)=0\quad \forall h\in M^{\rho^*}.$$
This means that $\int \int_0^T h_t(\sw)d\xi_t\nu(d\sw)=0$ for all such $h$, which implies that $\xi$ is indeed a $(\nu,\G)$-martingale. Since the bracket of the canonical process is the identity under $\gamma$, this is inherited by $\xi$ by Girsanov theorem under the assumption $\nu\ll\gamma$, so by Levy's theorem $\xi$ is then a $(\nu,\G)$-Brownian motion. By Lemma~\ref{lem: causal wiener equivalence}, $(\xi,id)_*\nu$ is causal. In light of the finiteness in \eqref{eq integ alpha}, this proves that the primal problem \eqref{eq: finiteness} is finite. 
The lower semicontinuity of $(\omega,\sw)\mapsto \int_0^T\rho(\abs{\sw_t-\omega_t})dt$ was established in Lemma~\ref{lem aux Sobolev}, and so by Theorem~\ref{prop duality} we get that there is no duality gap. To conclude the proof, we only need to check the equality between \eqref{eq: finiteness} and the refined dual. For this, we rewrite \eqref{eq integ alpha} as
	$$\textstyle
	\int\int_0^T \rho(\alpha_t(\sw)) dt\,\nu(d\sw)= - \int\int_0^T  \rho^* (\hat{\zeta}_t(\sw)) dt\,\nu(d\sw) + \int\int_0^T\hat{\zeta}_t(\sw) d\sw_t\,\nu(d\sw),$$
	where we used that $\int\int_0^T \hat{\zeta}_t(\sw)d\xi_t\nu(d\sw)=0$. This proves that the refined dual has a greater value than the primal, and we conclude by Lemma~\ref{lem FY}.
\end{proof}

We can finally give the proof of Theorem~\ref{thm orlicz}:

\begin{proof}[Proof of Theorem \ref{thm orlicz}]
Under the assumptions made, both $\rho$ and $\rho^*$ are N-functions in $\Delta_2$. For Point $(i)$, and thanks to Lemma~\ref{lem FY}, we only need to prove that when the refined dual is finite, its value coincides with the primal one. This follows by Lemmata~\ref{lem attain} and \ref{lem many}. Point $(ii)$ is contained in Lemma~\ref{lem orlicz}. Point $(iii)$ follows from the latter lemma, and because the strict convexity of $\rho$ implies the differentiability of $\rho^*$.  Finally Point $(iv)$ is given by Lemma~\ref{lem many} together with an application of Girsanov theorem.
\end{proof}

We stress that Point $(iv)$ of Theorem~\ref{thm orlicz} can also be obtained via more sophisticated stochastic analysis arguments: if the primal problem (equiv.\ the refined dual) is finite, then as in the proof of Lemma~\ref{lem orlicz} one shows that for all $\textstyle F\in S_a(\G):\, \int\int_0^T F_t\, d\sw_t\,\nu(d\sw)\leq (v+1)\|F\|_{\rho^*}$, 
where $v$ is the common optimal value. If $|F|\leq C$ then $\|F\|_{\rho^*}\leq C\|1\|_{\rho^*}<\infty$. This suggests that $ \textstyle\{\int_0^T F_t\, d\sw_t: \,F\in S_a(\G),\, \|F\|_{\infty}\leq 1\} $
is bounded in $\nu$-probability, so by the Bichteler-Dellacherie theorem $\sw$ is a $(\nu,\G)$-semimartingale, and we conclude by Girsanov theorem. Such a proof is reminiscent of original arguments in \cite{Jac85}.

\section{Projections of processes}\label{app proj}
We recall the notions of (dual) optional and predictable projections, which are used throughout the article, and refer to \cite[Ch.\ VI]{DellacherieMeyerB} for an accurate study of the subject.

Let $X$ be a positive or bounded measurable process on a { filtered probability space} $(\Omega,\H,\PROB)$.
The \emph{optional projection} of $X$ is the unique (up to indistinguishability) optional process $Y$ such that
\[
\EE\left[X_\tau\IND_{\{\tau<\infty\}}|\H_\tau\right]=Y_\tau\IND_{\{\tau<\infty\}}\;\, \textrm{a.s.}
\]
for every stopping time $\tau$.
The \emph{predictable projection} of $X$ is the unique (up to indistinguishability) predictable process $N$ such that
\[
\EE\left[X_\tau\IND_{\{\tau<\infty\}}|\H_{\tau-}\right]=N_\tau\IND_{\{\tau<\infty\}}\;\, \textrm{a.s.}
\]
for every predictable stopping time $\tau$. These notions of projection can be given for a broader class of processes, including those of integrable variation; see \cite[Remark~VI.44-(f)]{DellacherieMeyerB}.

Now, let $H$ be a raw process of integrable variation on $(\Omega,\H,\PROB)$.
The \emph{dual optional (resp.\ predictable) projection} of $H$ is the optional (resp.\ predictable) integrable variation process $U$ defined by
\[\textstyle
\EE\left[\int_0^{\infty}X_tdU_t\right]=
\EE\left[\int_0^{\infty}X_tdH_t\right]
\]
for any bounded optional (resp.\ predictable) process $X$. W.l.o.g.\ we assume $U_0=0$.

The following lemma is fundamental for the proof of Theorem~\ref{thm: gWnu_gen} and  Proposition~\ref{prop:OS}. It follows directly from \cite[Lemma~7, Appendix~I]{DellacherieMeyerB}{, and holds for any two filtrations $\F$ and $\G$  on $\C$}.
\begin{lemma}
\label{lem proj}
Let $\pi\in \Pi(\mu, \nu)$ be a (non-necessarily causal) transport plan, and let $\Lambda$ be a $(\mathcal B([0,T])\otimes \F_T\otimes \G_T)$-measurable process on $\C\times\C$ of integrable variation. Then:
\begin{enumerate}
\item 
The optional projection of $\Lambda$ with respect to $(\pi,\{\F\otimes \{\emptyset,\C\}\}^{\pi})$ (resp. $(\pi,\{\{\emptyset,\C\}\otimes \G\}^{\pi})$), which we denote by $^o\Lambda^{\F}$ (resp. $^o\Lambda^{\G}$) is $\pi$-indistinguishable from an optional process with respect to $(\mu,\F^{\mu})$ (resp. $(\nu,\G^{\nu})$), so w.l.o.g.\ one may assume
\[
^o\Lambda^{\F}(\omega,\sw)=\, ^o\Lambda^{\F}(\omega),\qquad ^o\Lambda^{\G}(\omega,\sw)=\, ^o\Lambda^{\G}(\sw).
\]
The analogous statement holds for the predictable projections.
\item 
The dual optional projection of $\Lambda$ with respect to $(\pi,\{\F\otimes \{\emptyset,\C\}\}^{\pi})$ (resp. $(\pi,\{\{\emptyset,\C\}\otimes \G\}^{\pi})$), which we denote by\ $^{\F}\Lambda^o$ (resp. $^{\G}\Lambda^o$), is $\pi$-indistinguishable from an optional process with respect to $(\mu,\F^{\mu})$ (resp. $(\nu,\G^{\nu})$), so w.l.o.g.\ one may assume \[
^{\F}\Lambda^o(\omega,\sw)=\, ^{\F}\Lambda^o(\omega),\qquad ^{\G}\Lambda^o(\omega,\sw)=\, ^{\G}\Lambda^o(\sw).
\]
The analogous statement holds for the dual predictable projections.
\end{enumerate}
\end{lemma}

\bibliographystyle{amsalpha}
\bibliography{biblio_causal_enlargement}

\providecommand{\bysame}{\leavevmode\hbox to3em{\hrulefill}\thinspace}
\providecommand{\MR}{\relax\ifhmode\unskip\space\fi MR }
\providecommand{\MRhref}[2]{%
  \href{http://www.ams.org/mathscinet-getitem?mr=#1}{#2}
}
\providecommand{\href}[2]{#2}
\begin{thebibliography}{GHLT14}

\bibitem[ADI06]{ADI}
S.~Ankirchner, S.~Dereich, and P.~Imkeller, \emph{The {S}hannon information of
  filtrations and the additional logarithmic utility of insiders}, Ann. Probab.
  \textbf{34} (2006), no.~2, 743--778. \MR{2223957}

\bibitem[ADI07]{ADIGirsanov}
\bysame, \emph{Enlargement of filtrations and continuous {G}irsanov-type
  embeddings}, S\'eminaire de {P}robabilit\'es {XL}, Lecture Notes in Math.,
  vol. 1899, Springer, Berlin, 2007, pp.~389--410. \MR{2409018}

\bibitem[AIS98]{AIS}
J.~Amendinger, P.~Imkeller, and M.~Schweizer, \emph{Additional logarithmic
  utility of an insider}, Stochastic processes and their applications
  \textbf{75} (1998), no.~2, 263--286.

\bibitem[AL16]{AL16}
A.~Aksamit and L.~Li, \emph{Projections, pseudo-stopping times and the
  immersion property}, arXiv:1409.0298v3, 2016.

\bibitem[Ald81]{Aldousweak}
D.~Aldous, \emph{Weak convergence and the general theory of processes (weak
  convergence of stochastic processes for processes viewed in the strasbourg
  manner)}, unpublished, 1981.

\bibitem[BBLZ16]{BBLZ}
J.~Backhoff, M.~Beiglb\"ock, Y.~Lin, and A.~Zalashko, \emph{Causal transport in
  discrete time and applications}, Submitted, arXiv:1606.04062, 2016.

\bibitem[BC77]{BaxterChacon}
J.~R. Baxter and R.~V. Chacon, \emph{Compactness of stopping times}, Z.
  Wahrscheinlichkeitstheorie und Verw. Gebiete \textbf{40} (1977), no.~3,
  169--181. \MR{0517871}

\bibitem[BCH16]{BCH}
M.~Beiglb\"{o}ck, A.~Cox, and M.~Huesmann, \emph{Optimal transport and
  skorokhod embedding}, To appear Inventiones Mathematicae, arXiv:1307.3656v4,
  2016.

\bibitem[BG14]{BGriessler}
M.~Beiglb\"{o}ck and C.~Griessler, \emph{An optimality principle with
  applications in optimal transport}, Submitted, arXiv:1404.7054v2, 2014.

\bibitem[BHLP13]{BHP13}
M.~Beiglb{\"o}ck, P.~Henry-Labord\'ere, and F.~Penkner, \emph{Model-independent
  bounds for option prices -- a mass transport approach}, Finance Stoch.
  \textbf{17} (2013), no.~3, 477--501.

\bibitem[BS11]{MathiasWalter}
M.~Beiglb{\"o}ck and W.~Schachermayer, \emph{Duality for {B}orel measurable
  cost functions}, Trans. Amer. Math. Soc. \textbf{363} (2011), no.~8,
  4203--4224. \MR{2792985 (2012k:49108)}

\bibitem[BY78]{BY78}
P.~Br{\'e}maud and M.~Yor, \emph{Changes of filtrations and of probability
  measures}, Zeitschrift f{\"u}r Wahrscheinlichkeitstheorie und verwandte
  Gebiete \textbf{45} (1978), no.~4, 269--295.

\bibitem[CK92]{CvitKaratzas}
J.~Cvitani{\'c} and I.~Karatzas, \emph{Convex duality in constrained portfolio
  optimization}, Ann. Appl. Probab. \textbf{2} (1992), no.~4, 767--818.
  \MR{1189418}

\bibitem[CT07]{CoquetToldo}
F.~Coquet and S.~Toldo, \emph{Convergence of values in optimal stopping and
  convergence of optimal stopping times}, Electron. J. Probab. \textbf{12}
  (2007), no. 8, 207--228. \MR{2299917}

\bibitem[DM80]{DellacherieMeyerB}
C.~Dellacherie and P.-A. Meyer, \emph{Probabilit\'es et potentiel. {C}hapitres
  {V} \`a {VIII}}, revised ed., Actualit\'es Scientifiques et Industrielles,
  vol. 1385, Hermann, Paris, 1980, Th{\'e}orie des martingales. [Martingale
  theory]. \MR{566768}

\bibitem[F{\"U}04]{FeyelUstunel}
D.~Feyel and A.~S. {\"U}st{\"u}nel, \emph{Monge-{K}antorovitch measure
  transportation and {M}onge-{A}mp\`ere equation on {W}iener space}, Probab.
  Theory Related Fields \textbf{128} (2004), no.~3, 347--385. \MR{2036490
  (2004m:60121)}

\bibitem[GHLT14]{GHT14}
A.~Galichon, P.~Henry-Labord\'ere, and N.~Touzi, \emph{A stochastic control
  approach to no-arbitrage bounds given marginals, with an application to
  lookback options}, The Annals of Applied Probability \textbf{24} (2014),
  no.~1, 312--336.

\bibitem[Jac80]{Jac81}
J.~Jacod, \emph{Weak and strong solutions of stochastic differential
  equations}, Stochastics \textbf{3} (1980), 171--191.

\bibitem[Jac85]{Jac85}
\bysame, \emph{Grossissement initial, hypoth\`ese ({H}'), et th\'eor\`eme de
  {G}irsanov}, Grossissements de Filtrations: Exemples et Applications
  (T.~Jeulin and M.~Yor, eds.), Lecture Notes in Mathematics, vol. 1118,
  Springer, Berlin - Heidelberg, 1985, pp.~15--35.

\bibitem[Jeu80]{Jeu80}
T.~Jeulin, \emph{Semi-martingales et grossissement d'une filtration}, Lecture
  Notes in Mathematics, vol. 833, Springer, Berlin, 1980. \MR{MR604176
  (82h:60106)}

\bibitem[JY78]{JY78}
T.~Jeulin and M.~Yor, \emph{Grossissement d'une filtration et semi-martingales:
  formules explicites}, S\'eminaire de Probabilit\'es, XII (Univ. Strasbourg,
  Strasbourg, 1976/1977), Lecture Notes in Math., vol. 649, Springer, Berlin,
  1978, pp.~78--97. \MR{MR519998}

\bibitem[JY79]{JY79}
\bysame, \emph{In\'egalit\'e de {H}ardy, semimartingales, et faux-amis},
  S\'eminaire de probabilit\'es de Strasbourg \textbf{13} (1979), 332--359
  (fre).

\bibitem[JYC09]{JYC09}
M.~Jeanblanc, M.~Yor, and M.~Chesney, \emph{Mathematical methods for financial
  markets}, Springer Finance, Springer-Verlag London Ltd., London, 2009.
  \MR{2568861}

\bibitem[Kan42]{Ka42}
L.~V. Kantorovich, \emph{On the transfer of masses}, Dokl. Akad. Nauk. SSSR,
  vol.~37, 1942, pp.~227--229.

\bibitem[Kel84]{KellererDuality}
H.~G. Kellerer, \emph{Duality theorems for marginal problems}, Z. Wahrsch.
  Verw. Gebiete \textbf{67} (1984), no.~4, 399--432. \MR{761565}

\bibitem[KLS87]{KarLehShr}
I.~Karatzas, J.~P. Lehoczky, and S.~E. Shreve, \emph{Optimal portfolio and
  consumption decisions for a ``small investor'' on a finite horizon}, SIAM J.
  Control Optim. \textbf{25} (1987), no.~6, 1557--1586. \MR{912456}

\bibitem[KLSX91]{KaratLehShreXu}
I.~Karatzas, J.~P. Lehoczky, S.~E. Shreve, and G.-L. Xu, \emph{Martingale and
  duality methods for utility maximization in an incomplete market}, SIAM J.
  Control Optim. \textbf{29} (1991), no.~3, 702--730. \MR{1089152}

\bibitem[KP15]{KchiaProtter}
Y.~Kchia and P.~Protter, \emph{Progressive filtration expansions via a process,
  with applications to insider trading}, Int. J. Theor. Appl. Finance
  \textbf{18} (2015), no.~4, 1550027, 48. \MR{3358108}

\bibitem[Kur14]{Kurtz2014}
T.~G. Kurtz, \emph{Weak and strong solutions of general stochastic models},
  Electron. Commun. Probab. \textbf{19} (2014), no. 58, 16. \MR{3254737}

\bibitem[Las15]{Lassalle2}
R.~Lassalle, \emph{Causal transference plans and their {M}onge-{K}antorovich
  problems}, Submitted, arXiv:1303.6925.v2, 2015.

\bibitem[L{\'e}o12]{Leo_Girsanov}
C.~L{\'e}onard, \emph{Girsanov theory under a finite entropy condition},
  S\'eminaire de {P}robabilit\'es {XLIV}, Lecture Notes in Math., vol. 2046,
  Springer, Heidelberg, 2012, pp.~429--465. \MR{2953359}

\bibitem[LP90]{LambertonPages}
D.~Lamberton and G.~Pag{\`e}s, \emph{Sur l'approximation des r\'eduites}, Ann.
  Inst. H. Poincar\'e Probab. Statist. \textbf{26} (1990), no.~2, 331--355.
  \MR{1063754}

\bibitem[Mon84]{Mo81}
G.~Monge, \emph{M\'emoire sur la th\'eorie des d\'eblais et des remblais,
  histoire de l'acad\'emie royale des sciences ann\'ee 1781}, Avec les Memoires
  de Mathematique \& de Physique, pour la m\^eme Annee](2e partie)(1784)
  Histoire (1784), 34--38.

\bibitem[MY06]{MY06}
R.~Mansuy and M.~Yor, \emph{Random times and enlargements of filtrations in a
  brownian setting}, Springer, 2006.

\bibitem[Pfl09]{Pflug}
G.~Ch. Pflug, \emph{Version-independence and nested distributions in multistage
  stochastic optimization}, SIAM Journal on Optimization \textbf{20} (2009),
  no.~3, 1406--1420.

\bibitem[PK96]{PikovKaratz}
I.~Pikovsky and I.~Karatzas, \emph{Anticipative portfolio optimization}, Adv.
  in Appl. Probab. \textbf{28} (1996), no.~4, 1095--1122. \MR{1418248}

\bibitem[PP12]{PflugPichler}
G.~Ch. Pflug and A.~Pichler, \emph{A distance for multistage stochastic
  optimization models}, SIAM J. Optim. \textbf{22} (2012), no.~1, 1--23.
  \MR{2902682}

\bibitem[Pro04]{Pro04}
P.~Protter, \emph{Stochastic integration and differential equations}, 2.1 ed.,
  Applications of Mathematics (New York), Springer-Verlag, Berlin, 2004.
  \MR{MR1037262 (91i:60148)}

\bibitem[RR91]{RaoRen}
M.~Rao and Z.~Ren, \emph{Theory of {O}rlicz spaces}, Monographs and Textbooks
  in Pure and Applied Mathematics, vol. 146, Marcel Dekker, Inc., New York,
  1991. \MR{1113700}

\bibitem[Sio58]{Sion}
M.~Sion, \emph{On general minimax theorems}, Pacific J. Math. \textbf{8}
  (1958), 171--176. \MR{0097026}

\bibitem[Vil03]{Villani}
C.~Villani, \emph{Topics in optimal transportation}, no.~58, American
  Mathematical Soc., 2003.

\bibitem[XS92]{XuShrI}
G.-L. Xu and S.~E. Shreve, \emph{A duality method for optimal consumption and
  investment under short-selling prohibition. {I}. {G}eneral market
  coefficients}, Ann. Appl. Probab. \textbf{2} (1992), no.~1, 87--112.
  \MR{1143394}

\bibitem[Yor85]{Yor85}
M.~Yor, \emph{Entropie d'une partition, et grossissement initial d'une
  filtration}, Grossissements de filtrations: exemples et applications,
  Springer, 1985, pp.~45--58.

\bibitem[Yor97]{Yor97}
\bysame, \emph{Some aspects of {B}rownian motion. {P}art {II}}, Lectures in
  Mathematics ETH Z\"urich, Birkh\"auser Verlag, Basel, 1997, Some recent
  martingale problems. \MR{1442263}

\bibitem[YW71]{YW}
T.~Yamada and S.~Watanabe, \emph{On the uniqueness of solutions of stochastic
  differential equations}, Journal of Mathematics of Kyoto University
  \textbf{11} (1971), no.~1, 155--167.

\bibitem[Zae15]{Zaev}
D.~Zaev, \emph{On the {M}onge--{K}antorovich {P}roblem with {A}dditional
  {L}inear {C}onstraints}, Mat. Zametki \textbf{98} (2015), no.~5, 664--683.
  \MR{3438523}

\end{thebibliography}

\end{document}